\documentclass[12pt, reqno]{amsart}
\usepackage[utf8]{inputenc}
\usepackage{mathtools}
\usepackage{amsmath}
\usepackage{amssymb}
\usepackage{graphicx}
\usepackage{changepage}
\graphicspath{ {images/} }
\usepackage{mathrsfs}
\usepackage{listings}
\usepackage{amsthm}
\usepackage[toc,page]{appendix}
\usepackage[margin=0.9in]{geometry}
\usepackage{framed,enumitem}
\usepackage{amsaddr}
\usepackage{newtxtext}
\usepackage{bbm}

\usepackage[colorlinks = true,
linkcolor = blue,
urlcolor  = blue,
citecolor = blue,
anchorcolor = blue]{hyperref}

\usepackage{thmtools, thm-restate}

\newtheorem{thm}{Theorem}[section]
\newtheorem{defn}[thm]{Definition}
\newtheorem{cor}[thm]{Corollary}

\newtheorem{prop}[thm]{Proposition}
\newtheorem{lem}[thm]{Lemma}
\newtheorem{rk}[thm]{Remark}

\title{Ergodicity results for the open KPZ equation}
\author{Shalin Parekh}
\date{November 2023}

\begin{document}

\maketitle

\begin{abstract}
We give a new proof of existence as well as two proofs of uniqueness of the invariant measure of the open-boundary KPZ equation on $[0,1]$, for all possible choices of inhomogeneous Neumann boundary data. 
Both proofs yield an exponential convergence result in total variation when combined with the strong Feller property which was recently established in \cite{KM22}. An important ingredient in both proofs is the construction of a compact state space for the Markov operator to act on, which measurably contains all of the usual Hölder spaces modulo constants. Along the way, the strong Feller property is extended to this larger class of initial conditions. 
The arguments do not rely on exact descriptions of the invariant measures, and some of the results generalize to the case of a spatially colored noise and other boundary conditions.
\end{abstract}

\section{Introduction}

The open KPZ equation with inhomogeneous Neumann boundary condition is a stochastic PDE on the domain $(t,x) \in \Bbb R_+\times [0,1]$ formally written as 
\begin{equation}\label{kpz0}\partial_t h(t,x) = \partial_x^2 h(t,x) +(\partial_x h(t,x))^2 +\xi(t,x),\;\;\;\;\;\; \partial_x h(t,0)=A+\frac12,\;\;\;\;\; \partial_x h(t,1)=B-\frac12,
\end{equation}
where $A,B\in \Bbb R.$ Here $\xi$ is a space-time white noise, the Gaussian field on $\Bbb R_+\times [0,1]$ with formal covariance $\Bbb E[\xi(t,x)\xi(s,y)] = \delta(t-s)\delta(x-y)$. This equation appears as a scaling limit of various probabilistic systems with very simple microscopic dynamics, for instance ASEP with open boundaries \cite{CS18, Parekh}. The equation \eqref{kpz0} is ill-posed due to the singularity of the noise, and so one usually works with the KPZ equation using the Hopf-Cole transform: if $h$ solves \eqref{kpz0} then $z:=e^h$ formally solves the Robin-boundary stochastic heat equation with multiplicative noise:
\begin{equation}\partial_t z(t,x) = \partial_x^2 z(t,x) + z(t,x)\xi(t,x),\;\;\;\;
\partial_x z(t,0) =Az(t,0), \;\;\;\; \partial_x z(t,1) = Bz(t,1),\label{mshe0}
\end{equation}
which is a well-posed Itô SPDE when written in mild form. The extra additive constant 1/2 in both parameters of \eqref{kpz0} is a convention that has been taken in order to take into account certain ``boundary-renormalization'' effects that become apparent when applying the Hopf-Cole transform, see \cite[Remark 1.11]{GH17} and \cite{GPS}. Recently other methods such as regularity structures \cite{Hai14, GH17} have been successful in making sense of \eqref{kpz0} directly, which is relevant to this work although we will mainly work with the Hopf-Cole transform here.

In a recent work \cite{Corwin-Knizel} proves existence of an invariant measure for \eqref{kpz0} for all choices of boundary parameters $A,B\in \Bbb R.$ They prove this by taking limits of known invariant measures for ASEP and showing that they are tight under the KPZ scaling considered in \cite{CS18,Parekh}. The invariant measures are nontrivial and have explicit descriptions in terms of certain functionals of Brownian motions or Brownian bridges for certain specific values of $A,B$. When $A-B>-1$ the authors characterize the law in terms of a Laplace transform using the dual Hahn process. In the present work we give an independent proof of existence of an invariant measure for each $A,B \in \Bbb R$, which also works for other type of Gaussian noises that are colored in space and white in time (see Corollary \ref{existence}).

In a follow-up paper \cite{KM22} prove the strong Feller property for \eqref{kpz0}, following a method developed by \cite{HM17} using the theory of regularity structures to make sense of KPZ. This made it possible to prove uniqueness of the invariant measure for certain specific choices of $A,B$, assuming one knows certain exact information about the invariant measure in advance, in particular that it has full support. The difficulty with proving uniqueness for general parameters is that we do not have explicit descriptions of the laws of the invariant measures for general values of the boundary parameters. There is a conjecture from \cite{BLD21} but it remains to be shown that the resultant object is in fact invariant when $A-B\leq -1$. We also mention \cite{BW10,BW18} for more about the formulaic viewpoint on the invariant measures.

The present work extends the uniqueness result to general values of $A,B$ without relying on any exact descriptions of the invariant measures, thus completing a final step in the proof of the uniqueness conjecture from \cite{Cor20, Corwin-Knizel}, actually we will go much further and prove that there is a spectral gap and uniform mixing over all initial data. We give two separate proofs, one via a support theorem and another via a polymer coupling argument. In the support proof, the main tools that we use are the Feynman-Kac formula, the Cameron-Martin theorem, and the strong Feller property on $C[0,1]$. In the polymer argument, the main tool used is the convolution property for the propagators of the stochastic heat equation \eqref{mshe0}, which also gives a one-force one solution principle for the open KPZ equation.

To describe our results we need to specify the choice of topology. Let $\mathcal X_0$ denote the quotient of the space of continuous functions $[0,1]\to \Bbb R$ by the equivalence relation $f\sim g$ if $f-g$ is constant. Denote by $[f]$ the equivalence class of $f$. For $[f],[g]\in\mathcal X_0$ we define their distance $$d([f],[g]) := d_{\text{Proh}} \bigg( \frac{e^f}{\int_{[0,1]}e^f},\frac{e^g}{\int_{[0,1]}e^g}\bigg),$$
where $d_{\text{Proh}}$ is a Prohorov metric on the space of probability measures on $[0,1]$ (i.e., any metric which metrizes the weak topology), satisfying $d_{\text{Proh}}(\mu,\nu)\leq \|\mu-\nu\|_{TV}$. Finally we denote $\mathcal X$ to be the completion of $\mathcal X_0$ under the metric $d$, which is isometric to the space of probability measures on $[0,1]$ (with Prohorov distance) and is therefore compact. Intuitively the elements of $\mathcal X$ can be viewed as formal logarithms of the measures that are initial conditions for the Hopf-Cole transform \eqref{mshe0}, modulo addition of constants. Thus we still use ``$[\mu]$'' to denote elements of the completion. Sometimes we also use a more generic ``$\phi$'' to denote elements of $\mathcal X$ in situations when it is not particularly important to emphasize that elements are equivalence classes associated to a Borel measure on $[0,1]$.

We will generally take the convention of formulating results in terms of discrete-time Markov semigroups, in order to simplify the exposition and also to be consistent with \cite{HM17, KM22}. This does not lose any generality. For a measurable function $F:\mathcal X\to \Bbb R$ and an initial condition $\phi \in \mathcal X$ the Markov operator $\mathfrak P$ for the open KPZ equation acts by the formula $\mathfrak PF(\phi):= \Bbb E_\phi [ F([h(1,\cdot)])].$
Here $\Bbb E_\phi$ denotes the expectation with respect to the measure $\Bbb P_\phi$ on the canonical space $\mathcal X^\Bbb N$ that corresponds to starting the process at initial state $\phi$, and $h(1,\cdot)$ denotes the spatial process $x\mapsto h(1,x)$ with $h$ solving the open KPZ equation \eqref{kpz0}.  We extend the results of \cite{KM22} by proving the following.

\begin{thm}[Uniqueness via a support theorem]\label{mr}Fix boundary parameters $A,B\in \Bbb R$. The Markov operator $\mathfrak P$ for the open KPZ equation is globally defined and strong Feller on the compact state space $\mathcal X.$ Furthermore there exists a unique invariant measure $\rho_{A,B}$ for the semigroup. The invariant measure has full support on $\mathcal X$. Moreover, for every deterministic initial condition $\phi \in \mathcal X$ we have the ergodic theorem: $$\Bbb P_\phi\bigg(\lim_{N\to\infty}\frac1N \sum_{n=1}^N F([h(n,\cdot)]) = \int_\mathcal X F\;d\rho_{A,B}\bigg)=1,$$
for all bounded measurable $F:\mathcal X\to \Bbb R.$
\end{thm}

This result will still be further strengthened in Theorems \ref{1f1s} and \ref{tvc} below. The existence statement is proved in Corollary \ref{existence} below, independently of all previous works. The uniqueness and ergodicity result is proved as Corollary \ref{mr2} below. The proof of the uniqueness result uses the results of \cite{KM22} as an input, namely that the strong Feller property already holds in \textit{some} (stronger) topology. Note in the above theorem that the ergodic theorem holds for \textit{every} initial condition (as opposed to \textit{almost} every initial condition), which rules out the existence of another invariant measure even if one restricts the action of the semigroup to some finer metric space $\mathcal Y$ which embeds continuously into $\mathcal X$ (e.g. the space $\mathcal Y$ of Hölder continuous functions given in \eqref{yspace} below).

The proof of our theorem is done by showing full support in $\mathcal X$ of the open KPZ equation at any fixed time. To show full support, we use the Cameron-Martin theorem and the Feynman-Kac formula. In particular we use the Hopf-Cole transform and rely on the linear structure of the equation. In fact, throughout the whole paper, we work with only with the Hopf-Cole transform, using the results of \cite{HM17,KM22} as given and then building off of those results using only classical techniques without any further use of regularity structures.

When there exists a natural coupling of the Markov process started from different initial data, it is nice to have a ``one-force-one-solution'' principle where one can show that the effect of the initial data vanishes uniformly. Therefore we will give a second approach to proving the uniqueness result without using the strong Feller property, and this approach also gives an exponential ``one-force one-solution principle'' for open KPZ. To formulate the result, we let $\mathcal X$ be as above. The following result will be proved \textit{independently} of Theorem \ref{mr}.

\begin{thm}[Uniqueness via polymer coupling]\label{1f1s}
The following ``one-force one-solution'' principle holds for open KPZ: fix a space-time white noise on $(-\infty,0]\times[0,1]$ and let $h_N^\mu,h_N^\nu: [-N,0]\times [0,1]\to \Bbb R$ denote the solutions to open KPZ driven by $\xi$ such that their Hopf-Cole transforms are started from deterministic nonnegative Borel measures $\mu,\nu$ at time $t=-N$. Then 
there exist $C,c>0$ such that for all $N \ge 1$ one has
$$\Bbb E \bigg[ \sup_{[\mu],[\nu]\in\mathcal X} d_\mathcal X\big([h_N^\mu(0,\cdot)],[h_N^\nu(0,\cdot)]\big)\bigg] \leq Ce^{-cN}.$$ 
\end{thm}

This will be proved as Theorem \ref{1f1s2} in the main body. Our approach to proving this result follows \cite[Section 4]{GK20} which is done in the periodic case, and uses the convolution property of \eqref{mshe0} in a manner inspired by the approach of \cite{Sinai, EKMS}. In particular, the proof of Theorem \ref{1f1s} still works to show uniqueness in the case of a more general Gaussian noise that is colored in space and white in time, and does \textit{not} rely on the strong Feller property.

By Markov's inequality and Borel-Cantelli, the expectation bound also implies almost sure convergence as $N\to \infty$. We actually prove something stronger, namely that a uniform synchronization result holds in the (much stronger) Hölder metric $d_\mathcal Y$ defined in \eqref{yspace} below, more precisely for $\kappa\in (0,1/2)$, \begin{equation}\label{sync}\lim_{N\to \infty}\sup_{[\mu],[\nu]\in \mathcal X}\big\|(h_N^\mu(0,\cdot)-h_N^\mu(0,0))-(h_N^\nu(0,\cdot) - h_N^\nu(0,0))\big\|_{C^\kappa[0,1]} = 0,\;\;\;\;\;a.s..
\end{equation}See 
Corollary \ref{1f1sy} for the proof, which is based on combining Theorem \ref{1f1s} with a compactness trick. We remark that the proof of that corollary goes through in the case of a spatially smooth noise or a periodic boundary condition as well, thus recovering some results of \cite{Rosati} with a different argument based on compactness and the convolution property. In case one is only interested in the synchronization result, its proof is fairly brief and self-contained within Definition \ref{cpk}, Lemma \ref{relaxation}, Theorem \ref{1f1s2}, and Corollary \ref{1f1sy}.

We now explain the differences of obtaining Theorem \ref{1f1s} compared to method of \cite{Rosati}. That work contains an abstract result which guarantees the one-force-one-solution principle for SPDEs whose solutions are linear as a function of the initial data like the one studied here, as long as Assumption 4.2 therein is satisfied. In \cite{Rosati} Assumption 4.2 is verified for periodic boundary conditions and space-time white noise using paracontrolled distributions, but that is not the only possible way of doing so and Assumption 4.2 for \eqref{mshe0} likely follows easily from the results of Section 2 of the current paper. While that result is not completely independent of the Hopf-Cole transform, its reliance on it is less crucial and it can be applied to more general equations. Furthermore, that paper derives slightly stronger synchronization results than the one we prove here, for instance in that paper \eqref{sync} is shown to occur exponentially fast.

We now turn to the question of total variation convergence. We actually give two proofs that there is a spectral gap, one using the support theorem and the other using the one-force-one-solution principle. The following result will be proved as Theorem \ref{tv3} and then independently reproved as Corollary \ref{bob} below. Both proofs use the strong Feller property and compactness as an input.

\begin{thm}[Spectral gap with a uniform multiplying constant]\label{tvc}
Let $\mathfrak P_*$ denote the adjoint Markov operator for open KPZ, which acts on the space of measures on $\mathcal X$. Then there exist $C,c >0$ such that for all $N$ one has $$ \sup_{\phi\in\mathcal X} \| \mathfrak P_*^N\delta_\phi - \rho_{A,B}\|_{TV}\leq Ce^{-cN},$$ where $\delta_\phi$ denotes a Dirac mass at the initial condition $\phi\in \mathcal X$.
\end{thm}
For a Markov process, the bound in Theorem \ref{tvc} is an extremely strong form of geometric ergodicity, where the exponential decay rate and the multiplying constant are both uniform over all possible choices of initial data (in the total variation norm). Since our state space $\mathcal X$ measurably contains many other important spaces of possible initial conditions, e.g. Hölder spaces such as $\mathcal Y$ defined in \eqref{yspace} below, the result also implies the corresponding result with the sup over $\mathcal X$ replaced by a sup over $\mathcal Y$ and the TV norm on $\mathcal M(\mathcal X)$ replaced by the TV norm on $\mathcal M(\mathcal Y)$.

Theorem \ref{tvc} then leads to a number of other interesting questions. In particular one can then sensibly ask about the optimal values of the constants $c$ and $C$, which is equivalent to computing (respectively) the \textit{spectral gap} and the \textit{mixing time} of the process. The most interesting formulation of this question would be to study the dependence of both the spectral gap and the mixing time on the size of the interval and the values of the boundary parameters $A,B$. We do not pursue these optimality questions in the present work. In the case of TASEP and ASEP, the analogous question has recently been studied in great depth \cite{SS,Sch,ES,GNS}. In those papers the authors observe some very interesting cutoff phenomena, where for certain boundary parameters there is a very explicit cutoff for the mixing time as a function of the boundary parameters. Recent work of \cite{HS} has even calculated the exact profile of the cutoff, showing that it is given by the cdf of the Tracy-Widom distribution. We expect that there may be similar results for the open KPZ equation as well as the periodic KPZ equation, though this is a difficult question far beyond the scope of this work.

In Theorem \ref{tvc}, we remark that our use of the strong Feller property and compactness tricks to prove the spectral gap seems similar in spirit to works of \cite{TW} which used such ideas to prove a spectral gap for the dynamic $\Phi_2^4$ model. We further remark that the analogous result for the \textit{additive} stochastic heat equation is false. That is, if we consider \eqref{kpz0} without the nonlinear term on the right side, then although the constant $c$ would still be finite, the other constant $C$ would become infinite. In other words, there is still a spectral gap but there is no mixing time \textit{uniform over all initial conditions}. And in fact the state space \textit{cannot} be compactified for the linearized equation. The finiteness of the supremum in Theorem \ref{tvc} and the existence of a compact state space seems related to the so-called ``coming down from infinity'' property that was proved in \cite{MW} for the dynamic $\Phi_3^4$ model. See Remark \ref{solflow} for further details.

Another interesting problem for future work is to see what can be proved about the invariant measures on the half space as opposed to a bounded interval. A conjectural characterization of the extremal invariant measures was given in \cite{BC22}, but it remains to be proved. A compact state space or a spectral gap is unlikely in that case, but it may be possible to show uniqueness using some variant of the arguments given here. For KPZ on the whole line, a synchronization result was recently considered and proved in \cite{JRS} using a different approach with Busemann functions, which seems to be the only synchronization result to date for an SPDE on an unbounded interval driven by Gaussian space-time white noise. Results such as \cite{BCK, Ba13} were some of the first works to prove synchronization in a noncompact setting, though these had a Poissonian kick-forcing.

Let us now say a word about the proofs of the above three theorems, and the ordering of the logic. As stated above, Theorems \ref{mr} and \ref{1f1s} will first be derived completely independently of one another. The proofs of both theorems will ultimately leverage particular properties of \eqref{mshe} that interpret $z$ as the partition function of a polymer. Once these two theorems are proved, then we will show that Theorem \ref{tvc} can be deduced from either of these first two theorems. To go from either of Theorem \ref{mr} or \ref{1f1s} to Theorem \ref{tvc}, we will derive some abstract results about Markov processes on compact state spaces. Specifically, we have the following two general results.

\begin{thm}[Automatic criteria for uniform spectral gap on compact spaces]

\begin{enumerate}
\item Suppose $(\mathfrak P_t)_{t\ge 0}$ is a strongly continuous semigroup of Markov operators on a compact metric space $\mathcal X$. Assume that $\mathfrak P_t$ is strong Feller for each $t> 0$ and admits a unique invariant probability measure $\rho$. Then there exist $C,c>0$ such that for all $t\ge 0$ 
 one has $ \sup_{\phi\in\mathcal X} \big\| \rho - \mathfrak P_t^*\delta_\phi\big\|_{TV}\leq Ce^{-ct},$
where $\mathfrak P_t^*$ is the adjoint Markov semigroup, and $\delta_\phi$ denotes a Dirac mass at the initial condition $\phi\in \mathcal X$. 

\item Let $\mathfrak P$ be a strong Feller Markov operator on some compact state space $\mathcal X$, and let $\mathfrak P_*$ denote the adjoint operator on measures. Assume that 
$\mathfrak P^N_*\delta_\phi - \mathfrak P^N_*\delta_\psi$ converges weakly to $0$ for all $\phi,\psi\in\mathcal X$, i.e., $\lim_{N\to\infty}|\mathfrak P^NF(\phi) - \mathfrak P^NF(\psi)|=0$ for all $\phi,\psi\in\mathcal X$ and all Lipschitz continuous $F:\mathcal X\to \Bbb R$. Then there exists $C,c>0$ such that for all $N\in \Bbb N$ one has $\sup_{\phi,\psi\in\mathcal X} \|\mathfrak P^N_*\delta_\phi - \mathfrak P^N_*\delta_\psi\|_{TV} \leq Ce^{-cN}.$
\end{enumerate}
\end{thm}
The first result states that in \textit{continuous time}, compactness plus the strong Feller property automatically allows us to upgrade unique ergodicity to the seemingly much stronger property of a spectral gap with uniform multiplying constant. It is clear that Theorem \ref{tvc} can be deduced immediately from Theorem \ref{mr} using this first result. In discrete time the analogous statement is simply not true even on finite state spaces. However the second result states that even in discrete time, compactness plus the the strong Feller property allows us to upgrade pointwise convergence of differences in the weak topology to a spectral gap with uniform multiplying constant. It is clear that Theorem \ref{tvc} is immediate from Theorem \ref{1f1s} using this second result.

The monograph of Meyn-Tweedie-Glynn \cite[Chapter 16]{MT} contains many abstract and powerful results for Markov processes on general state spaces, however despite the extensive literature we were not able to find the above two results in efficient and readily-applicable forms. Since these results are evidently quite useful in problems concerning KPZ and related areas, we have included short and self-contained proofs of these abstract results, independent of one another and different from the arguments of \cite{MT}. See Theorems \ref{tv3} and \ref{expo} below. 
\\
\\
\textbf{Outline:} Section 2 recalls some standard facts about the Robin-boundary stochastic heat equation that will be useful going forward. Section 3 will prove Theorem \ref{mr}. Section 4 will prove Theorem \ref{1f1s}. The methods and results of Sections 3 and 4 should be viewed as disjoint from each other, and both sections contain an independent proof of Theorem \ref{tvc}.
\\
\\
\textbf{Acknowledgements:} I thank Yu Gu, Ivan Corwin, and Konstantin Matetski for very helpful discussions and for reading portions of the draft. I thank Ran Tao for finding some mistakes and typos.

\section{Properties of the Robin SHE and well-posedness of its Markov semigroup on the quotient space $\mathcal X$}

In this section we derive some bounds on the Robin-boundary stochastic heat equation that are standard 
but will be very important in later sections. We also prove existence of an invariant measure. The results described here are still valid for noises that white in time but colored in space, with some minor modification of the proofs. Some other boundary conditions such as Dirichlet (zero) conditions on one endpoint or both can also be considered (these correspond to $A=\infty$ or $B=\infty)$, with the only exception that the strict positivity result (Proposition \ref{nega}) will not hold in case of the latter.

Consider the solution of the multiplicative stochastic heat equation with Robin boundary condition on $[0,1]$:

\begin{equation}\partial_t z = \partial_x^2 z + z\xi,\;\;\;\;
\partial_x z(t,0) =Az(t,0), \;\;\;\; \partial_x z(t,1) = B z(t,1).\label{mshe}
\end{equation}

Here (and henceforth) $A,B$ are just some fixed real numbers. The formal definition of such a process is a $\mathcal F^\xi$-adapted process $z(t,\cdot)$ such that a.s. for all $t\ge0$ and $x\in [0,1]$ one has 
\begin{equation}z(t,x) = \int_{[0,1]} p^{A,B}_t (x,y) z_0(dy)+\int_0^t\int_{[0,1]} p^{A,B}_{t-s}(x,y) z(s,y) \xi(ds,dy) ,\label{mild}
\end{equation}where the initial condition $z_0$ is a fixed deterministic tempered distribution supported on $[0,1]$, and the latter is an Itô integral with $p^{A,B}$ the Robin heat kernel \cite{CS18,Parekh}, that is, the fundamental solution of the heat equation whose boundary conditions at $x=0$ and $x=1$ match those of \eqref{mshe}.

We are interested in the invariant measures for this process. However, there is no \textit{bona fide} invariant measure since $cz$ is a solution for all $c>0$, whenever $z$ is. Consequently the invariant measures must be infinite measures and cannot be probability measures. 

We fix this problem by constructing state spaces $\mathcal X$ and $\mathcal Y$ consisting of \textit{equivalence classes} of functions, which is where $z(t,\cdot)$ will live. More precisely, we define $\mathcal X$ to be the set of all finite non-negative non-zero Borel measures on $[0,1]$ modulo the relation $\mu \sim c\mu$ for all $c>0$. For equivalence classes $[\mu],[\nu]\in \mathcal X$ we define their distance through the metric 
\begin{equation}\label{xspace}d_\mathcal X([\mu],[\nu]):= d_{\text{Proh}} \bigg( \frac{\mu}{\mu[0,1]}, \frac{\nu}{\nu[0,1]}\bigg),
\end{equation}
where $d_{\text{Proh}}$ is a Prohorov metric on the space of probability measures on $[0,1]$ (i.e., any metric which metrizes the weak topology), satisfying $d_{\text{Proh}}(\mu,\nu)\leq \|\mu-\nu\|_{TV}$. An explicit example of such a metric $d_{\text{Proh}}$ is given by $$d_{\text{Proh}}(\mu,\nu) = \sup_{\|g\|_{C^1[0,1]} \leq 1} \bigg|\int_{[0,1]} g \;d(\mu -\nu)\bigg|,$$
where $\|g\|_{C^1[0,1]}:= \|g\|_{C[0,1]}+\|g'\|_{C[0,1]}.$ An important fact about $\mathcal X$ is that it is compact (since any collection of probability measures on $[0,1]$ is automatically tight), which will be a very useful simplifying tool later. 

Now fix some $\kappa\in(0,1/2)$. For separability reasons, we take the convention that $C^\kappa[0,1]$ is defined as the \textit{closure of smooth functions} inside the usual Banach space of $\kappa$-Holder continuous functions $[0,1]\to \Bbb R$. Now we define another metric space $\mathcal Y$ to be the set of all strictly positive functions on $[0,1]$ lying in $C^{\kappa}[0,1]$, modulo the equivalence relation $f\sim cf$ for all $c>0$. We define their distance through the metric 
\begin{equation}\label{yspace}d_\mathcal Y([f],[g]):= \bigg\| \log(f)-\log(g) - \int_{[0,1]} \big(\log(f)-\log(g)\big)\bigg\|_{C^{\kappa}[0,1]}.\end{equation}
For $d_\mathcal Y$ any equivalent metric could also be used, such as $\big\| \log(f)-\log(g) - \big(\log(f(0))-\log(g(0))\big)\big\|_{C^{\kappa}}.$ One may easily verify that this is a complete metric space and that it embeds densely into $\mathcal X$ by identifying a continuous function $f$ with the measure $\mu(A) = \int_A f(x)dx.$ Certainly closed bounded subsets in $\mathcal Y$ are still closed in $\mathcal X$, and are thus compact in $\mathcal X$. This already implies that all closed sets (hence all open sets) in $\mathcal Y$ are Borel measurable in $\mathcal X$, which will be important in the next section.

We will next establish global well-posedness of the Markov semigroup of \eqref{mshe} on the space of finite Borel measures on $[0,1]$, and show that the process $z(1,\cdot)$ already takes values in $C^\kappa$ almost surely for $\kappa<1/2$. Furthermore we will derive a bound on the expected $C^{\kappa}$ norm which depends only on the total mass of the initial data, which will in particular allow us to prove (in Proposition \ref{upgrade} and Corollary \ref{SFP}) that $[z(t,\cdot)]$ is strong Feller on $\mathcal X$, using the knowledge from \cite{KM22} that it is strong Feller on $\mathcal Y$.

We will need to use some purely deterministic facts about the Robin heat kernel which we summarize without proof here, see e.g. \cite[Section 2.5]{Fri} for a reference on how to obtain these facts:
\begin{enumerate}
    \item For each $x\in [0,1]$ the function $(t,y)\mapsto p_t^{A,B}(x,y)$ = $p_t^{A,B}(y,x)$ solves the heat equation on $\Bbb R_+\times [0,1]$ with initial data $\delta_x$, and furthermore $\partial_y|_{y=0} p_t^{A,B}(x,y) = Ap_t^{A,B}(x,0) $ and $\partial_y|_{y=1} p_t^{A,B}(x,y) = Bp_t^{A,B}(x,1) $ for all $x,t$. It is a symmetric kernel in the sense that $p_t^{A,B}(x,y) = p_t^{A,B}(y,x).$

    \item One may write \begin{equation}p_t^{A,B}(x,y) = p_t^{neu}(x,y) \mathbf E_{Br[0,t]}^{x\to y}\bigg[ e^{-AL_0^W(t) - BL_1^W(t)}\bigg],\label{fk1}
    \end{equation}
    where $p_t^{neu}$ is the Neumann heat kernel $(A=B=0)$ and the expectation is with respect to a Brownian bridge of diffusive rate 2 on the time interval $[0,t]$ starting from $x$ and ending at $y$ that is reflected at the boundaries of $[0,1]$ (so that it always stays inside the interval). Here $L_a^W(t)$ equals the local time of $W$ at spatial location $a$ up to time $t$. This equation reflects a Feynman-Kac decomposition for the Robin boundary heat equation, where the boundary condition is reinterpreted as a special case of \textit{Neumann }boundary heat equation, but with a multiplicative forcing term consisting of two Dirac masses of weights $-A$ and $-B$ at the respective boundaries $x=0$ and $x=1$. In particular \eqref{fk1} implies the \textit{strict positivity} of the kernel: $p_t^{A,B}(x,y)>0$ for all $A,B\in\Bbb R$, all $x,y\in[0,1]$ and all $t>0$. 

    \item Similarly to the ordinary heat kernel, for any integers $k_1,k_2\ge 0$ one has the following bound uniformly over all $(x,y,t)\in [0,1]^2\times(0,T]$
    \begin{equation}
        \partial_t^{k_1}\partial_x^{k_2}p_t^{A,B}(x,y)=\partial_x^{2k_1+k_2}p_t^{A,B}(x,y) \leq C (|x-y|+t^{1/2})^{-(1+2k_1+k_2)},\label{p=p0+e}
    \end{equation}
    where $C$ may depend on $k_1,k_2,T,A,B$ but not on $x,y,t$. Since $ \partial_y^2p_t(x,y)=\partial_x^2p_t(x,y)$, the bound \eqref{p=p0+e} is actually equivalent to a seemingly more general bound $\partial_t^{k_1}\partial_x^{k_2}\partial_y^{k_3}p_t^{A,B}(x,y) \leq C (|x-y|+t^{1/2})^{-(1+2k_1+k_2+k_3)}$ for $k_3$ even. 
    Because of these bounds, one can hope to sensibly define $$(P^{A,B}\phi)(t,x):=\int_{[0,t)\times [0,1]} p_{t-s}^{A,B}(x,y) \phi(s,y)dyds$$ for certain tempered space-time distributions $\phi$ supported on $[0,\infty)\times [0,1].$ 
    One would intuitively expect $p^{A,B}$ to improve space-time regularity by two exponents because the singularity of $p^{A,B}$ along the space-time diagonal (with respect to the parabolic metric) is of order $-1$ by \eqref{p=p0+e}, whereas the parabolic dimension of the space-time is $-3$.
    Indeed let us define $C_\mathfrak s^\alpha([0,1]\times[0,T])$ to be all those tempered distributions (supported on that rectangle) that are of parabolic Hölder regularity of exponent $\alpha\in\Bbb R$, more precisely the completion of smooth functions on $[0,1]\times[0,T]$ with respect to the norm given by $$\|f\|_{C^{\alpha,\sigma}_\mathfrak s([0,1]\times[0,T])}:= \sup_{(t,x)\in[0,1]\times[0,T]} \sup_{\lambda\in (0,1]} \sup_{\varphi \in B_r} \lambda^{-\alpha}(f,S^\lambda_{(t,x)}\varphi)$$ where $(f,\varphi)$ denotes the pairing in $L^2([0,1]\times[0,T])$, where the scaling operators are defined by $S^\lambda_{(t,x)}\phi (s,y) = \lambda^{-3}\phi(\lambda^{-2}(t-s),\lambda^{-1}(x-y)),$ and where if $\alpha<0$ then $r:=-\lfloor\alpha\rfloor$ and $B_r$ is the set of all smooth functions of $C^r(\Bbb R^2)$-norm less than 1 with support contained in the unit ball of $\Bbb R^2$, and if $\alpha>0$ then $r=\lceil \alpha \rceil$ and $B_r$ is the set of all smooth functions of $C^0$ norm less than 1, supported on the unit ball of $\Bbb R^2$ that are orthogonal in $L^2(\Bbb R^2)$ to all polynomials of parabolic degree less than or equal to $r$. Then we have the \textit{Schauder estimate} which says that $P^{A,B}$ is a bounded linear operator from $C_\mathfrak s^{\alpha}$ to $C_\mathfrak s^{\alpha+2}$ for all $\alpha <1$ ($\alpha\notin\Bbb Z)$, i.e., there exists $C=C(\alpha)>0$ such that \begin{equation}\label{schauder}\|P^{A,B}\phi\|_{C_\mathfrak s^{\alpha+2}([0,1]\times[0,T])} \leq C \|\phi\|_{C_\mathfrak s^{\alpha}([0,1]\times[0,T])},\end{equation}
    uniformly over all $\phi \in C_\mathfrak s^{\alpha}([0,1]\times[0,T]).$ The proof of the claim \eqref{schauder} follows from a standard and elementary argument using \eqref{p=p0+e} and following the definition of the Holder spaces, noting e.g. that for very small $\lambda$ the operation $P^{A,B}$ applied to $S^\lambda_{(t,x)}\varphi$ behaves at worst like $\lambda^2 S^\lambda_{(t,x)}\psi$ for some other smooth function $\psi$. Although we will not need to use \eqref{schauder} for $\alpha>1$ ($\alpha\notin \Bbb Z$), it is still true in that regime except that the bound only holds for $\phi$ in the closed linear subspace of $C_\mathfrak s^{\alpha}([0,1]\times[0,T])$ such that $\phi$ and all of its space-time derivatives respect the Robin boundary conditions (this allows for integration by parts without worrying about appearance of boundary terms). 
    Another corollary of \eqref{p=p0+e} is that for any (purely spatial) Scwhartz distribution $\varphi$ supported on $[0,1]$ we can also sensibly define the purely spatial convolution at fixed time $t$ as $$p_t^{A,B}\varphi(x) := \int_{[0,1]} p_t^{A,B}(x,y)\varphi(dy).$$
\end{enumerate}

\begin{thm}[Existence/Hölder estimates of Mild Solutions]\label{70} Fix some deterministic signed Borel measure $\mu$ on $[0,1]$ and let $|\mu|$ denote the sum of its positive and negative parts. Let $\xi$ be a space-time white noise on $[0,1]\times[0,T]$. There exists a unique $C[0,1]$-valued process $(z(t,\cdot))_{t \in(0, T]}$ which is adapted to $\mathcal F_t:=\sigma( \{\xi(t,\cdot) \}_{s\leq t})$, and satisfies \eqref{mild} with $z_0=\mu$. 

For $p\ge 1$ this mild solution satisfies
\begin{equation}\label{bound1}\Bbb E[|z(t,x)|^p]^{2/p}\leq Ct^{-1/2}\|\mu\|_{TV}(p_t^{A,B}|\mu|)(x)\leq Ct^{-1}\|\mu\|_{TV}^2,
\end{equation}
for some universal constant $C=C(A,B,p,T)$ not depending on $\mu$ or $(t,x)\in [0,T]\times[0,1]$. For $\kappa>0$ we furthermore have the existence of a universal constant $C=C(A,B,\kappa,p,T)>0$ independent of $\mu$ and $\epsilon \in (0,T)$ such that 
\begin{equation}\label{bound2}\Bbb E\big[\|z\|_{C_\mathfrak s^{1/2-\kappa}([0,1]\times[\epsilon,T])}^p\big]^{1/p} \leq C \epsilon^{-\frac34} \|\mu\|_{TV}. \end{equation}
Finally if $\mu_n$ is a sequence of Borel measures on $[0,1]$ converging weakly to $\mu$ and if $z_{\mu_n},z_{\mu}$ denote the associated solutions, then for all $\epsilon \in (0,T)$ we have the convergence in probability 
\begin{equation}\label{conv}\lim_{n\to\infty}\Bbb P\big( \|z_{\mu_n}-z_{\mu}\|_{C_\mathfrak s^{1/2-\kappa}([0,1]\times[\epsilon,T])}>\delta\big) = 0,\;\;\;\;\;\text{for all  } \delta>0.
\end{equation}
		
\end{thm}

\begin{proof} 
Throughout the proof we abbreviate $p_t^{A,B}$ as simply $p_t$ and $p_t^{A,B}\mu$ as $p_t\mu$. Define a sequence of iterates for $t\leq T$ and $x \in [0,1]$: \begin{align*}u_0(t,x)&:= (p_t\mu)(x) , \\ u_{n+1}(t,x)&:= \int_0^t \int_{[0,1]} p_{t-s}(x,y) u_n(s,y) \xi(dy,ds).\end{align*}
Note that $|p_t\mu(x)|\leq (p_t|\mu|)(x)$. To show existence of a solution it suffices to show that $\sum u_n$ converges in the appropriate Banach space, specifically we choose the space of $\mathcal F^\xi_t$-adapted processes $z$ such that $\sup_{x\in[0,1],s\leq T} s^{1/2} \Bbb E[|z(s,y)|^p]^{1/p}<\infty$. Now we define $$f_n(t):= \sup_{\substack{x\in [0,1] \\ s \in [0,t]}} s^{1/2} (p_s|\mu|(x))^{-1} \Bbb E [ |u_n(s,x)|^p]^{2/p}$$
where the RHS is defined as $+\infty$ if the stochastic integral defining $u_n$ fails to exist. For the proof it will be crucial that $(p_s|\mu|)(x)>0$ for all $s>0$ and $x\in[0,1]$ which is true as long as $\mu$ is not the zero measure.
		
Using the Itô isometry, the definition of $f_n$, and the fact that $p_{t-s}^{A,B}(x,y) \leq C_T(t-s)^{-1/2}$ on $[0,1]\times[0,T]$ (see equation \eqref{p=p0+e} or alternatively
\cite[Proposition 3.31]{Parekh}), we compute by applying Burkholder-Davis-Gundy and then Minkowski's inequality (in that order) \begin{align*}\Bbb E [|u_{n+1}(t,x)|^p]^{2/p} &\leq C_p\Bbb E\bigg[ \bigg( \int_0^t \int_{[0,1]} p_{t-s}(x,y)^2 u_n(s,y)^2dyds\bigg)^{p/2}\bigg]^{2/p}\\ &\leq C_p \int_0^t \int_{[0,1]} p_{t-s}(x,y)^2 \Bbb E[u_n(s,y)^p]^{2/p} dyds \\ &\leq C_p\int_0^t \int_{[0,1]} p_{t-s}(x,y)^2 \cdot s^{-1/2} (p_s|\mu|)(y) f_n(s) dy ds \\ &\leq C \int_0^t (t-s)^{-1/2}s^{-1/2} \bigg[\int_{[0,1]} p_{t-s}(x,y) (p_s|\mu|)(y) dy \bigg]f_n(s) ds \\ & = C (p_t|\mu|)(x) \int_0^t (t-s)^{-1/2}s^{-1/2}  f_n(s)ds \end{align*}
where we used the semigroup property in the final line. Multiplying both sides by $t^{1/2}(p_t|\mu|(x))^{-1}$, we find that \begin{align*}t^{1/2}(p_t|\mu|(x))^{-1} \Bbb E[|u_{n+1}(t,x)|^p]^{2/p} &\leq Ct^{1/2} \int_0^t (t-s)^{-1/2}s^{-1/2}f_n(s)ds .
\end{align*}
Notice from the definition that $f_n$ is an increasing function, therefore (by making a substitution $s=tu)$ one may see that the RHS of the last expression is an increasing function of $t$. It follows by taking a supremum that \begin{align*}f_{n+1}(t) &\leq Ct^{1/2} \int_0^t (t-s)^{-1/2}s^{-1/2} f_n(s)ds \end{align*} which we can iterate twice to obtain $$f_{n+2}(t) \leq C't^{1/2} \int_0^t s^{-1/2}f_n(s)ds.$$ 
Again using the fact that $p_t(x,y)\leq C_Tt^{-1/2}$, we obtain
\begin{equation}\label{tv1}(p_t|\mu|)(x)\leq C_T \|\mu\|_{TV} t^{-1/2}
\end{equation}
which implies that $\sup_{t\in [0,T]}f_0(t)\leq C_T\|\mu\|_{TV}$. Then using this bound we can iterate this recursion to obtain $$f_n(t) \lesssim \|\mu\|_{TV} t^{n/2}/(n/2)! $$ 
We have just proved that $$\Bbb E[|u_n(t,x)|^p]^{2/p} \leq C\|\mu\|_{TV}(p_t|\mu|)(x) t^{(n-1)/2}/(n/2)!$$ for a constant $C$ not depending on $t \in [0,T]$, $x\in [0,1]$, or $n \in \Bbb N$.
Since $z=\sum_n u_n$, we find that $$\Bbb E[|z(t,x)|^p]^{1/p} \leq \sum_{n\ge 0} \Bbb E[|u_n(t,x)|^p]^{1/p} \leq C'\|\mu\|_{TV}^{1/2}t^{-1/4} p_t|\mu|(x)^{1/2}$$ where $C'$ is obtained by summing the infinite series (bounding $t$ by the terminal time $T$), which which will not depend on $\mu,t,x$. This already proves the first inequality in \eqref{bound1}. The second inequality in \eqref{bound1} just comes by once again using the same bound \eqref{tv1} again. 

The proof of uniqueness is very similar. In fact, the difference of two mild solutions started from the same initial data would itself be a mild solution started from $\mu=0$, thus by \eqref{bound1} must be zero a.s. for all $(t,x)$. By the continuity statement which we prove next, the entire profile $(t,x)\mapsto z(t,x)$ must then be zero almost surely.


Now let us move onto the proof of \eqref{bound2}. Consider any smooth nonnegative test function $\phi(t,x)$ supported in $[\epsilon,T]\times [0,1]$. By applying Burkholder and then Minkowski (in that order) we find that 
\begin{align*}\Bbb E\bigg[\bigg| \int_{[0,1]\times [\epsilon,T]} \phi(s,y) z(s,y)\xi(ds,dy)\bigg|^p \bigg]^{2/p}&\leq C_p \Bbb E\bigg[\bigg( \int_{[0,1]\times [\epsilon,T]} \phi(s,y)^2 z(s,y)^2dsdy\bigg)^{p/2} \bigg]^{2/p}\\ &\leq C_p \int_{[0,1]\times [\epsilon,T]} \phi(s,y)^2 \Bbb E[|z(s,y)|^p]^{2/p} dsdy\\ &\stackrel{\eqref{bound1}}{\leq} C\|\mu\|_{TV}^2 \int_{[0,1]\times [\epsilon,T]} \phi(s,y)^2 s^{-1} dsdy\\ &\leq C\|\mu\|_{TV}^2 \epsilon^{-1} \|\phi\|_{L^2([0,1]\times [\epsilon,T])}^2,
\end{align*}
where $C$ is independent of $\mu$ and $\epsilon.$ 

For $\lambda>0$ and $(t,x)\in\Bbb R^2$ define the scaling operator $S^\lambda_{(t,x)} : \mathcal S(\Bbb R^2) \to \mathcal S(\Bbb R^2)$ by $$\big(S^\lambda_{(t,x)}\phi\big)(s,y) := \lambda^{-3}\phi(\lambda^{-2}(s-t),\lambda^{-1}(y-x)).$$ The previous bound can now be interpreted as 
\begin{equation}\Bbb E\big[|(z\xi,S^\lambda_{(t,x)}\phi)|^p\big]^{1/p} \leq C\epsilon^{-1/2}\|\mu\|_{TV} \|S^\lambda_{(t,x)}\phi\|_{L^2}.\label{Hölderbound}
\end{equation}
We now notice that $\|S^\lambda_{(t,x)}\phi\|_{L^2}\leq C\lambda^{-3/2}$ independently of $(t,x) \in [\epsilon,T]\times[0,1]$ and independently of nonnegative smooth $\phi$ supported on $[\epsilon,T]\times[0,1]$ integrating to one, and independently of $\epsilon$. Using \cite[Theorem 2.7]{Chandra-Weber} this implies that $z\xi$ can be realized as a random element of $C^{-3/2-\kappa}_\mathfrak s([\epsilon,T]\times[0,1])$ and moreover $$\Bbb E [ \|z\xi\|_{C^{-3/2-\kappa}_\mathfrak s([\epsilon,T]\times[0,1])}^p]^{1/p}\leq C'\epsilon^{-1/2}\|\mu\|_{TV} .$$
Here $C'$ only depends on the choice of $\kappa>3/p$ and the constant $C$ appearing in \eqref{Hölderbound}, and is therefore independent of $\epsilon,\mu$. Letting $z_{\epsilon/2}(x):= z(\epsilon/2,x)$, we also have by Minkowski that for all space-time test functions $\phi$
\begin{align*}\Bbb E\bigg[\bigg|\int_{[0,1]} \phi(y) z(\epsilon/2,y)dy\bigg|^p\bigg]^{1/p} &\leq \int_{[0,1]} \phi(y)\Bbb E[|z(\epsilon/2,y)|^p]^{1/p}dy \\ &\stackrel{\eqref{bound1}}{\leq} C\int_0^1 \phi(y) \epsilon^{-1/2} \|\mu\|_{TV} dy \\&\leq C' \|\mu\|_{TV}\epsilon^{-1/2}\|\phi\|_{L^1[0,1]},\end{align*} 
so that by a similar reasoning we have the bound that $\Bbb E[\|z_{\epsilon/2} \|_{C^{-\kappa}[0,1]}^p]^{1/p} \leq C'\epsilon^{-1/2}\|\mu\|_{TV},$ where $C$ may depend on the choice of $\kappa<-1/p$ but not $\epsilon$ or $\mu.$ 

For $\kappa>0$ one defines the space $C^{-\kappa}[0,1]$ to be the closure of smooth functions on $[0,1]$ with respect to a norm given by $\sup_{x\in[0,1]}\sup_{\lambda\in (0,1]} \sup_{\phi\in B_r} \lambda^{-\alpha} (f,A^\lambda_x\phi)_{L^2[0,1]}$ where $A^\lambda_x\phi(y) = \lambda^{-1} \phi(\lambda^{-1}(y-x))$ and $B_r$ is the set of smooth functions supported on $[-1,1]$ of $C^r(\Bbb R)$-norm less than 1 with $r:=-\lceil \kappa\rceil$. Now by the smoothing effect of the Robin heat kernel, we have a deterministic bound $ \|p_t z_{\epsilon/2}\|_{C^{1/2-\kappa}[0,1]} \leq Ct^{-1/4} \|z_{\epsilon/2}\|_{C^{-\kappa}[0,1]}$, which is a consequence of \eqref{p=p0+e} (see also \cite[Section 4.3]{Hai09} for another proof via analytic semigroups). 
In particular by setting $t=\epsilon/2$ and taking expectation we have $$ \Bbb E[\|p_{\epsilon/2}z_{\epsilon/2}\|_{C^{1/2-\kappa}[0,1]}^p]^{1/p} \leq C\epsilon^{-3/4} \|\mu\|_{TV}.$$ 
where $C$ is independent of $\epsilon$ or $\mu$. Now let us write for $t>\epsilon$ 
$$z(t,x) = p_{t-\epsilon}(p_{\epsilon/2}z_{\epsilon/2})(x) + \int_{\epsilon/2}^t \int_{[0,1]} p_{t-s}(x,y)z(s,y)\xi(dy,ds)=:z_1^\epsilon(t,x)+z_2^{\epsilon}(t,x).$$ 
Now for $f\in C^{1/2-\kappa}[0,1]$ if we define $P_f(t,x)= p_tf(x)$ then one directly shows the bound (independently of $f$) that $\|P_f\|_{C_\mathfrak s^{1/2-\kappa}([0,T]\times[0,1])} \leq C_T \|f\|_{C^{1/2-\kappa}[0,1]},$ consequently we have that $$\Bbb E[\|z_1^\epsilon\|^p_{C_\mathfrak s^{1/2-\kappa}([\epsilon,T]\times[0,1])}]^{1/p} \leq C\Bbb E[\|p_{\epsilon/2}z_{\epsilon/2}\|_{C^{1/2-\kappa}[0,1]}^p]^{1/p} \leq C\epsilon^{-3/4} \|\mu\|_{TV}.$$ Now to bound $z_2$ we notice that $z_2$ is just the space-time convolution of the field $z\xi|_{[\epsilon/2,T]\times[0,1]}$ with the Robin heat kernel, therefore by the Schauder estimate \eqref{schauder} one has $$\Bbb E [ \|z_2\|_{C_\mathfrak s^{1/2-\kappa}([\epsilon,T]\times[0,1])}^p]^{1/p} \leq C \Bbb E [ \|z\xi\|_{C^{-3/2-\kappa}_\mathfrak s([\epsilon/2,T]\times[0,1])}^p]^{1/p}\leq C'\epsilon^{-1/2}\|\mu\|_{TV} ,$$ completing the proof of the claim \eqref{bound2}. 

Finally, to prove \eqref{conv} we will need to use the notion of \textbf{propagators}. For each fixed realization of $\xi$, let $z_t(a,x)$ denote the solution evaluated at $(t,x)$ when started from initial condition $\delta_a$. We now claim the existence of a modification of this process such that one has joint continuity in all three variables $(t,a,x),$ so long as $t$ is bounded away from $0$. This is easily justified by first noting that for fixed $x$, one has continuity of the process $(t,a)\mapsto z_t(x,a)$ e.g. by \eqref{bound2} with $\mu=\delta_x$, and secondly noting that one has an equality in the sense of finite-dimensional distributions $(z_t(a,x))_{a\in[0,1]} \stackrel{d}{=} (z_t(x,a))_{a\in[0,1]}$ for each \textit{fixed} $t>0$ and $x\in [0,1]$, which is clear by a time reversal of the noise on the interval $[0,t]$. Now suppose $\mu_n\to \mu$ weakly as Borel measures on $[0,1]$. By the uniform boundedness principle this implies that $\sup_{n\in\Bbb N} \|\mu_n\|_{TV}<\infty.$ Therefore, by \eqref{bound2} we have tightness of $(z_{\mu_n})_{n\ge 1}$ in the space $C_\mathfrak s^{1/2-\kappa}([\epsilon,T]\times[0,1]),$ for every $\epsilon,\kappa>0$. We just need to show therefore that $z_{\mu_n}(t,x)-z_\mu(t,x)$ converges to 0 in probability, for each \textit{individual} value of $(t,x)$. We claim that $$z_{\mu_n}(t,x) = \int_{[0,1]} z_{t}(a,x)\mu_n(da) \to \int_{[0,1]} z_{t}(a,x)\mu(da) = z_{\mu}(t,x).$$ The middle convergence statement is clear by continuity of the propagators in the $a$ variable, thus we just need to justify the two equalities. In other words we need to show that $\int_{[0,1]} z_{t}(a,x)\mu(da) = z_{\mu}(t,x)$ for all finite and nonnegative Borel measures on $[0,1].$ The left side is easily shown to satisfy the mild equation \eqref{mild}, thus by uniqueness of solutions must agree with the right side. 
\end{proof}

In the following proposition, $f$ and $\varphi$ will be functions defined on the unit interval $[0,1]$ and whenever we write $f*\varphi$ it should be understood as the periodic mollification $$\varphi*f(x) = \int_{[0,1]} f(y)\varphi\big((x-y)\text{mod 1}\big)dy.$$

\begin{prop}[Feynman-Kac formula for smooth noise]\label{mollif}
Let $\varphi: [0,1] \to \Bbb R$ be a smooth function such that $\int_{[0,1]} \varphi(u)du=1$. Assume that $\varphi$ is such that $f\mapsto \varphi * f$ is a contraction on $L^2[0,1]$ (which is always the case if $\varphi$ is non-negative). Define $\xi^{\varphi}(t,x)$ to be the spatial periodic convolution of $\xi$ by $\varphi$. Define the mild solution of the stochastic heat equation $z^\varphi$ driven by $\xi^\varphi$ exactly as in the Duhamel formula \eqref{mild} but with $\xi$ replaced by $\xi^{\varphi}$. Then the solution exists and satisfies both \eqref{bound1} and \eqref{bound2} as well as \eqref{conv} in Theorem \ref{70}. Moreover the constants appearing in those bounds are independent of $\varphi$. In particular the Markov semigroup for this equation is also globally defined on the space $\mathcal X$ from \eqref{xspace} and has at least one invariant measure.

Furthermore we may express the solution started from initial data $\mu$ as $$z^\varphi(t,y) = \int_{[0,1]} p_t^{neu}(x,y)\mathbf E_{Br[0,t]}^{x\to y}\bigg[ e^{\int_0^t \xi^\varphi(t-s,W_s)ds -\frac12(\varphi *\varphi)(0)t-AL_0^W(t) - BL_1^W(t)}\bigg]\mu(dx).$$ Here $p_t^{neu}$ is the Neumann heat kernel $(A=B=0)$ and $\mathbf E_{Br[0,t]}^{x\to y}$ is the expectation is with respect to a Brownian bridge of diffusive rate 2 on the time interval $[0,t]$ starting from $x$ and ending at $y$ that is reflected at the boundaries of $[0,1]$ (so that it always stays inside the interval), independent of $\xi$. Here $L_a^W(t)$ equals the local time of $W$ at spatial location $a$ up to time $t$. The quantity $\int_0^t \xi^\varphi(t-s,W_s)ds$ appearing in the exponential needs to be understood as a stochastic integral. 

Finally, if we take a sequence $\varphi_n$ of such functions such that $\|\varphi_n *f - f\|_{L^2}\to 0$ as $n\to\infty$ for all $f\in L^2[0,1]$, then for any fixed initial data $\mu$, one has convergence of the associated solutions $z^{\varphi_n}$ to the solution of \eqref{mshe} driven by space-time white noise. Convergence here is in probability with respect to the topology of $C^\kappa([0,1]\times[\epsilon,T])$ for all $\kappa\in (0,1/2)$ and all $\epsilon<T$. 
\end{prop}

Here is a rough sketch of the proof: first of all, the fact that $\varphi$ is a contraction on $L^2$ ensures that all of the bounds derived in the proof of Theorem \ref{70} go through unchanged if we replace $\xi$ by $\xi^\varphi$ (in particular the bounds on the expected $p^{th}$ moment of $u_n$). Then to prove the Feynman-Kac representation, one expands the exponential $e^{\int_0^t \xi^\varphi(t-s,W_s)ds -\frac12(\varphi *\varphi)(0)t}$ as an infinite sum of Hermite powers of the noise. Then by using the representation \eqref{fk1} of the Robin heat kernel, the $k^{th}$ term in the expansion may be expressed as a $k$-fold iterated integral of products of the $p_t^{A,B}$ against the mollified noise, which is precisely the mild solution as defined in \eqref{mild}. The convergence result stated at the end of the proposition can be proved by showing term-by-term convergence of the resulting chaos expansions as $n\to \infty$ (using e.g. the master theorem for polynomial chaos \cite[Theorem 2.3]{CSZ}). This would give pointwise convergence in probability, but it can automatically be upgraded to convergence in $C^\kappa$ since we have the bound \eqref{bound2} where the constant is independent of $\varphi$ as described in the above proposition statement. We also refer the reader to the work \cite{Bertini-Cancrini} for a similar theorem.

\begin{prop}[Convolution property]\label{convo}
Consider the propagators $(z_{s,t}(x,y))_{x,y\in[0,1],t\ge s\ge  0}$ where $(t,y) \mapsto z_{s,t}(x,y)$ solves the stochastic heat equation \eqref{mshe} with initial data $\delta_x$ at time $s$, all coupled to the same noise $\xi$. Then $$\int_{[0,1]} z_{s,t}(x,r) z_{t,u}(r,y)dr = z_{s,u}(x,y)$$ for all $x,y\in[0,1]$ and all $s<t<u$ almost surely. Furthermore, this property remains true if we replace $\xi$ by its spatial mollification $\xi^\varphi$ as in Proposition \ref{mollif}.
\end{prop}

The proof is the same as in the full line case, see \cite[Theorem 3.1(vii)]{AKQ11}.

\begin{prop}[Existence of negative moments]\label{nega}
Fix a space-time white noise $\xi$ on $[0,1]\times [0,T]$ and fix $A,B\in\Bbb R$. Consider the process $\mathbf{z}:=(z_t(x,y))_{x,y\in[0,1],t\ge 0}$ where $(t,y) \mapsto z_t(x,y)$ solves the stochastic heat equation \eqref{mshe} with initial data $\delta_x$. Then for all $p\ge 1$, all $\kappa<1/2$, and all $T>\epsilon>0$ we have that $$\Bbb E \bigg[\big\| \mathbf{z}\big\|_{C^{\kappa}_\mathfrak s([\epsilon,T]\times[0,1]^2)}^p+\big\| \mathbf{z}^{-1}\big\|_{C^{\kappa}_\mathfrak s([\epsilon,T]\times[0,1]^2)}^p \bigg]<\infty.$$ Here $\mathbf{z}^{-1} = 1/\mathbf{z}$ is the reciprocal of the process, to be interpreted as $+\infty$ if $z$ is not strictly positive.
\end{prop}

In particular we have strict positivity of the solution which is expected by classical results on the whole line such as such as \cite{Mueller}. Although such a result is standard, it seems that a proof has not been written down elsewhere for the Robin-boundary case (or even the Neumann case). Therefore we give a self-contained proof here for the sake of completeness, and also because some of the ideas in the proof (e.g. Feynman-Kac and approximation theory for the SHE) are instructive in understanding some of the proofs of the later sections as well. The proof will give an explicit tail decay on the above Holder norm proportional to $e^{-c(\log r)^2}$ for some $c>0$. Of course, $c$ will depend on the choice of $\epsilon,\kappa,A,B.$

\begin{proof}We break the proof into three steps.
\\
\\
\textit{Step 1.} In this step, we realize the solution $z$ of \eqref{mshe} as the limit of a two-parameter family of Banach-valued martingales. Consider the Cameron-Martin space $H:=L^2([0,T]\times[0,1])$ of the noise $\xi$. It will be convenient to write $L^2([0,T]\times[0,1]) = L^2[0,T]\otimes L^2[0,1] = : L^2_T\otimes L^2_X. $ We consider the orthonormal basis for $L^2_X$ given by $e_{2i}(x) = \sqrt{2}\sin(2\pi i(x-1/2))$ and $e_{2i+1}(x) = \sqrt{2}\cos(2\pi i(x-1/2))$ (here $i\ge 0$ but we exclude $e_0=0$), and we choose some orthonormal basis of smooth functions $\{f_j(t)\}_{j\ge 1}$ for $L^2_T$ whose exact form does not matter. 
With this in place, we now sample a space-time white noise on $[0,1]\times[0,T]$ by $$\xi = \sum_{j=1}^\infty \sum_{i=1}^\infty \xi_{ij} \; f_j\otimes e_i$$
where the $\xi_{ij}$ are IID $N(0,1)$. The infinite sum converges in some appropriately chosen Banach space whose norm is measurable on $H$ in the sense of L. Gross \cite{Gross}. We define the following hierarchy of approximations to $\xi$ for $k,n\in\Bbb N$:
$$\xi^k = \sum_{j=1}^\infty \sum_{i=1}^{2k} \xi_{ij} \; f_j\otimes e_i,$$
$$\xi^{n,k} = \sum_{j=1}^n \sum_{i=1}^{2k} \xi_{ij} \; f_j\otimes e_i.$$
We then define $\mathcal F_{n,k}$ and $\mathcal F_k$ to be the $\sigma$-algebras generated by $\xi^{n,k}$ and $\xi^k$, respectively. Note that $\xi^k$ can be written as the spatial mollification of $\xi$ where the mollifier (which is not necessarily positive or of unit $L^1$ norm but still satisfies the conditions of Proposition \ref{mollif}) equals a quotient of sine functions (the Dirichlet kernel).

Let $z_t(x,y)$ be as in the theorem statement. Because of the identity $\sum_{i=1}^{2k} e_i(x)^2 = 2k$ for all $x\in [0,1],$ one may show that $\Bbb E[z_t(x,y)|\mathcal F_k]$ is given by the solution $(t,y)\mapsto z_t^k(x,y)$ of the Ito equation $\partial_t z = \partial_x^2 z + z\xi^k$ started from $\delta_x$ (with the same boundary conditions). By Proposition \ref{mollif}, this can also be written as $$\Bbb E[z_t(x,y)|\mathcal F_k] = p_t^{neu}(x,y)\mathbf E_{Br[0,t]}^{x\to y} \bigg[ e^{\int_0^t \xi^k(t-s,W_s) ds - kt - AL_{0}^{W}(t) - BL_1^W(t)}\bigg]$$
where the expectation on the right side is with respect to a Brownian bridge on the time interval $[0,t]$ starting from $x$ and ending at $y$, reflected at the boundaries of $[0,1]$. Here $p_t^{neu}$ is the Neumann heat kernel on $[0,1]$. The integral $\xi^k(t-s,W_s)$ is well-defined as an Ito-integral for any continuous path $W$. Applying Fubini's theorem, we can also write 
\begin{align*}
z_t^{n,k}(x,y)&:=\Bbb E[z_t(x,y)|\mathcal F_{n,k}] \\&= \Bbb E[z_t^k(x,y)|\mathcal F_{n,k}] \\ &= p_t^{neu}(x,y)\mathbf E_{Br[0,t]}^{x\to y} \bigg[ \Bbb E\bigg[e^{\int_0^t \xi^k(t-s,W_s) ds}\bigg| \mathcal F_{n,k}\bigg] e^{- kt - AL_{0}^{W}(t) - BL_1^W(t)}\bigg] \\ &= p_t^{neu}(x,y)\mathbf E_{Br[0,t]}^{x\to y} \bigg[ e^{\int_0^t \xi^{n,k}(t-s,W_s) ds +\frac12 \sum_{i=1}^{2k} \sum_{j=n+1}^\infty \big(\int_0^tf_j(t-s)e_i(W_s)ds \big)^2- kt - AL_{0}^{W}(t) - BL_1^W(t)}\bigg] \\ &=p_t^{neu}(x,y)\mathbf E_{Br[0,t]}^{x\to y} \bigg[ e^{\int_0^t \xi^{n,k}(t-s,W_s) ds -\frac12 \sum_{i=1}^{2k}\sum_{j=1}^n  \big(\int_0^tf_j(t-s)e_i(W_s)ds \big)^2 - AL_{0}^{W}(t) - BL_1^W(t)}\bigg],
\end{align*}
where we used Gaussianity of the stochastic integrals (conditional on $W$) in the third equality and then orthonormality of $f_j$ in the final equality. The above expression is a martingale in the $n$ variable (for fixed $k$). By the Banach-valued martingale convergence theorem \cite[Theorem 1]{Ch}, we conclude that $z^k$ converges as $k\to \infty$ to $z$ in $C^\kappa([\epsilon,T]\times [0,1]^2)$ for all $\kappa<1/2$ and $\epsilon>0$, and we also conclude that (for each fixed $k\in \Bbb N)$ $z^{n,k}$ converges to $z^k$ as $n\to \infty$ in $C^\kappa([\epsilon,T]\times [0,1]^2)$. All convergences are almost sure and in $L^p$. Furthermore by Jensen's inequality, any conditional expectation necessarily contracts the $L^p$ norm, so we also have the following bounds for all $p\ge 1$:
\begin{align}\label{pbound}\sup_{n\ge 1} \Bbb E[z_t^{n,k}(x,y)^p] &\leq \Bbb E[z_t^k(x,y)^p],\;\;\;\forall k\in\Bbb N\\ \sup_{k\ge 1} \Bbb E[z_t^{k}(x,y)^p] &\leq \Bbb E[z_t(x,y)^p].\label{pbound2}
\end{align}
\textit{Step 2.} Now we establish the intermediate claim that 
\begin{equation}\label{negmombo}\sup_{(t,x,y) \in [\epsilon,T]\times[0,1]^2} \Bbb E [ z_t(x,y)^{-p} ] <\infty. 
\end{equation}
To do this, we will roughly follow the approach of \cite{HL} which in turn was inspired by \cite{MF}. Let $e_i,f_j$ be as in the previous step and define $$Q_{n,k}(s,t,x,y):= \sum_{j=1}^n\sum_{i=1}^{2k} f_{j}(s)f_{j}(t)e_{i}(x)e_{i}(y) = \Bbb E[ \xi^{n,k}(t,x)\xi^{n,k}(s,y)].$$ Then for continuous paths $W^1,W^2$ from $[0,t]\to[0,1]$, define $$\mathscr Q_{n,k} (t,W^1,W^2):= \int_0^t\int_0^t Q_{n,k}(t-s,t-r,W^1_r,W^2_s)drds.$$
Also define for continuous paths $W$ the quantity
$$\Gamma_t^{n,k}(W,\xi^{n,k}):=  e^{\int_0^t \xi^{n,k}(t-s,W_s)ds - \frac12 \mathscr Q_{n,k} (t,W,W)-AL_0^{W}(t) - BL_1^{W}(t)}.$$ By the results of Step 1, we can write $$z^{n,k}_t(x,y) = p_t^{neu}(x,y) \mathbf E_{Br[0,t]}^{x\to y} [\Gamma^{n,k}_t(W,\xi^{n,k})].$$ We now define a measure on paths by a Radon-Nikodym derivative with respect to the Brownian bridge measure from $x$ to $y$ reflected at the boundaries of $[0,1]$: $$\frac{d\mathbf P^{n,k,t}_{\xi,x,y} }{d\mathbf P_{Br[0,t]}^{x\to y} } (W) = \frac{\Gamma^{n,k}_t(W,\xi^{n,k})}{p_t^{neu}(x,y)^{-1} z^{n,k}_t(x,y)}.$$ This measure is well-defined because $\xi^{n,k}$ is a smooth function. Next, for all $\lambda>0$ and $(t,x,y) \in [0,\infty) \times [0,1]^2$ define the event $$A^{\lambda,n,k}_{t,x,y} := \big\{\xi: z_t^{n,k}(x,y)\geq \frac12 p_t^{A,B}(x,y) ,\;\;\;\; \big(\mathbf E_{\xi,x,y}^{n,k,t}\big)^{\otimes 2} [\mathscr Q_{n,k}(t,W^1,W^2)] \leq \lambda\big\},$$ where the expectation is with respect to two independent paths, each sampled from the path measure $\mathbf P^{n,k,t}_{\xi,x,y}$ defined above. We claim that there exists $\lambda>0$ such that \begin{equation}\label{lb}\inf_{n,k\ge 1} \inf_{x,y\in[0,1]} \inf_{t\in [\epsilon,T]} \Bbb P ( A^{\lambda,n,k}_{t,x,y})>0.
\end{equation}
To prove this, we write 
\begin{equation}\label{df1}\Bbb P( A^{\lambda,n,k}_{t,x,y}) \geq \Bbb P\big( z_t^{n,k}(x,y)\geq \frac12 p_t^{A,B}(x,y)\big) -\Bbb P ( B^{\lambda,n,k}_{t,x,y}),
\end{equation}
where $$B^{\lambda,n,k}_{t,x,y}:= \big\{\xi: z_t^{n,k}(x,y)\geq \frac12 p_t^{A,B}(x,y) ,\;\;\;\; \big(\mathbf E_{\xi,x,y}^{n,k,t}\big)^{\otimes 2} [\mathscr Q_{n,k}(t,W^1,W^2)] > \lambda\big\}. $$
The Paley-Zygmund inequality says that $\Bbb P(F>(1-\theta)\Bbb E[F]) \geq (1-\theta)^2 \Bbb E[F]^2/\Bbb E[F^2]$, so we have that 
\begin{equation}\label{df2}\Bbb P\big( z_t^{n,k}(x,y)\geq \frac12 p_t^{A,B}(x,y)\big) \geq \frac14 \frac{p_t^{A,B}(x,y)^2}{\Bbb E[z^{n,k}_t(x,y)^2]}\geq \frac14 \frac{p_t^{A,B}(x,y)^2}{\Bbb E[z_t(x,y)^2]},\end{equation}
where we use \eqref{pbound} and \eqref{pbound2} in the second inequality. Now to get an upper bound for $\Bbb P ( B^{\lambda,n,k}_{t,x,y}),$ notice that if $z_t^{n,k}(x,y)\geq \frac12 p_t^{A,B}(x,y)$, then $\big(\mathbf E_{\xi,x,y}^{n,k,t}\big)^{\otimes 2} [\mathscr Q_{n,k}(t,W^1,W^2)]$ is equal to
\begin{align*}&p_t^{neu}(x,y)^2 z_t^{n,k}(x,y)^{-2} \big(\mathbf{E}_{Br[0,t]}^{x\to y}\big)^{\otimes 2} \big[ \mathscr Q_{n,k}(t,W^1,W^2) \Gamma^{n,k}_t(W^1,\xi^{n,k})\Gamma^{n,k}_t(W^2,\xi^{n,k})] \\ &\leq 4p_t^{neu}(x,y)^2 p_t^{A,B}(x,y)^{-2} \big(\mathbf{E}_{Br[0,t]}^{x\to y}\big)^{\otimes 2} \big[ \mathscr Q_{n,k}(t,W^1,W^2) \Gamma^{n,k}_t(W^1,\xi^{n,k})\Gamma^{n,k}_t(W^2,\xi^{n,k})].
\end{align*}
Taking expectation $\Bbb E$ over the noise $\xi$ and applying Fubini gives 
\begin{align*}
\Bbb E&\bigg[\big(\mathbf{E}_{Br[0,t]}^{x\to y}\big)^{\otimes 2} \big[ \mathscr Q_{n,k}(t,W^1,W^2) \Gamma^{n,k}_t(W^1,\xi^{n,k})\Gamma^{n,k}_t(W^2,\xi^{n,k})]\bigg] \\ &\leq \big(\mathbf{E}_{Br[0,t]}^{x\to y}\big)^{\otimes 2} \big[ \mathscr Q_{n,k}(t,W^1,W^2) \Bbb E[\Gamma^{n,k}_t(W^1,\xi^{n,k})\Gamma^{n,k}_t(W^2,\xi^{n,k})]\big] \\ &=  \big(\mathbf{E}_{Br[0,t]}^{x\to y}\big)^{\otimes 2} \big[ \mathscr Q_{n,k}(t,W^1,W^2) e^{\mathscr Q_{n,k}(t,W^1,W^2) - A (L_0^{W^1}(t) + L_0^{W^2}(t))-B(L_1^{W^1}(t) + L_1^{W^2}(t))}\big],
\end{align*} where in the last equality we use the fact that $\sum_{i=1,2}\int_0^t \xi^{n,k}(t-s,W_s^i)ds$ is a Gaussian random variable conditional on $W^1,W^2,$ whose variance can be explicitly computed in terms of $\mathscr Q_{n,k}$. Now use the bound $a \leq e^a$ and we obtain that the last expression is bounded above by $$ \big(\mathbf{E}_{Br[0,t]}^{x\to y}\big)^{\otimes 2} \big[ e^{2\mathscr Q_{n,k}(t,W^1,W^2) - A (L_0^{W^1}(t) + L_0^{W^2}(t))-B(L_1^{W^1}(t) + L_1^{W^2}(t))}\big] =  p_t^{neu}(x,y)^{-2} \Bbb E[ z^{n,k}_{t,(\sqrt{2})}(x,y)^2],$$ where $z_{t,(\sqrt{2})}$ denotes the solution of the stochastic heat equation where the noise has an extra factor of $\sqrt{2}$ in the front, and $z^{n,k}_{t,(\sqrt{2})}$ is the respective approximation given in Step 1. In particular \eqref{pbound} and \eqref{pbound2} still hold and we have $\Bbb E[ z^{n,k}_{t,(\sqrt{2})}(x,y)^2] \leq \Bbb E[ z_{t,(\sqrt{2})}(x,y)^2]$ for all $n,k,x,y,t$. Summarizing these computations and applying Markov's inequality, we have 
\begin{align*}
\Bbb P(B^{\lambda,n,k}_{t,x,y}) &\leq \Bbb P\bigg(4p_t^{neu}(x,y)^2 p_t^{A,B}(x,y)^{-2}  \big(\mathbf{E}_{Br[0,t]}^{x\to y}\big)^{\otimes 2} \big[ \mathscr Q_{n,k}(t,W^1,W^2) \Gamma^{n,k}_t(W^1,\xi^{n,k})\Gamma^{n,k}_t(W^2,\xi^{n,k})]>\lambda \bigg) \\ &\leq 4p_t^{neu}(x,y)^2 p_t^{A,B}(x,y)^{-2} \lambda^{-1} \Bbb E \bigg[  \big(\mathbf{E}_{Br[0,t]}^{x\to y}\big)^{\otimes 2} \big[ \mathscr Q_{n,k}(t,W^1,W^2) \Gamma^{n,k}_t(W^1,\xi^{n,k})\Gamma^{n,k}_t(W^2,\xi^{n,k})] \bigg] \\ &\leq 4p_t^{A,B}(x,y)^{-2} \lambda^{-1} \Bbb E[z_{t,(\sqrt{2})}(x,y)^2 ],
\end{align*}
where we remark that all instances of $p_t^{neu}(x,y)$ have cancelled out. Note that $\lambda$ so far is independent of $n,k,x,y,t$, so now we exercise our freedom to set it equal to 
\begin{equation}\label{lambda}\lambda:= 32\bigg(\inf_{x,y\in[0,1]} \inf_{t\in [\epsilon,T]}  p_t^{A,B}(x,y)^{4} \Bbb E[z_{t,(\sqrt{2})}(x,y)^2 ]^{-1}\Bbb E[z_{t}(x,y)^2 ]^{-1}\bigg)^{-1}.
\end{equation}
The infimum is strictly positive because $\Bbb E[z_{t}(x,y)^2]\le Ct^{-1}$ by the first bound in Theorem \ref{70}. With this value of $\lambda$ we then see that $$ \Bbb P(B^{\lambda,n,k}_{t,x,y}) \leq \frac18 \frac{p_t^{A,B}(x,y)^2}{\Bbb E[z_t(x,y)^2]}$$ uniformly over all $n,k\ge 1$, $x,y\in [0,1]$, and $t\in [\epsilon,T].$ By combining this with \eqref{df1} and \eqref{df2}, this shows that $$\Bbb P(A^{\lambda,n,k}_{t,x,y}) \geq \frac18 \frac{p_t^{A,B}(x,y)^2}{\Bbb E[z_t(x,y)^2]} $$
uniformly over all $n,k\ge 1$, $x,y\in [0,1]$, and $t\in [\epsilon,T].$ This proves the claim \eqref{lb}, since by the first bound in Theorem \ref{70} we have $\Bbb E[z_t(x,y)^2] \leq C t^{-1}$ with $C$ independent of $t,x,y$. With this lower bound established on the probability of $A^{\lambda,n,k}_{t,x,y}$, we next claim that 
\begin{equation}\label{dist}z_t^{n,k}(x,y) \geq \frac12 p_t^{A,B}(x,y) e^{-\lambda^{1/2} dist_{L^2}(\xi, A^{\lambda,n,k}_{t,x,y})},
\end{equation}
for all $\lambda,t,x,y,n,k.$ In this expression we are viewing $\xi$ as lying in some abstract Wiener space and the quantity $dist_{L^2}(\xi,A)$ is defined as $\inf\{ \|\xi-\eta\|_{L^2([0,T]\times [0,1])} : \eta \in A\}, $ to be understood as $+\infty$ if $\xi-\eta\notin L^2([0,T]\times [0,1])$ for all $\eta$ in the abstract Wiener space. To prove this, write $z_t^{n,k}(x,y;\xi)$ to emphasize the dependence on the noise. Then note that for any $\eta \in A^{\lambda,n,k}_{t,x,y}, $ we have that 
\begin{align*}z_t^{n,k}(x,y;\xi) &= z_t^{n,k}(x,y;\eta) \mathbf E_{\eta,x,y}^{n,k,t} \bigg[ e^{\int_0^t \big[ \xi^{n,k}-\eta^{n,k} \big](t-s,W_s)ds} \bigg] \\ &\geq \frac12 p_t^{A,B}(x,y)e^{\mathbf E_{\eta,x,y}^{n,k,t} \big[\int_0^t \big[ \xi^{n,k}-\eta^{n,k} \big](t-s,W_s)ds\big]}.
\end{align*}
where we use Jensen's inequality and the definition of $A^{\lambda,n,k}_{t,x,y}$ in the second bound. Next we note by Cauchy-Schwarz that if $\eta \in A^{\lambda,n,k}_{t,x,y}, $ then
\begin{align*}
    \bigg|\mathbf E_{\eta,x,y}^{n,k,t} \big[&\int_0^t \big[ \xi^{n,k}-\eta^{n,k} \big](t-s,W_s)ds\big] \bigg|  = \bigg| \sum_{i=1}^n\sum_{j=1}^{2k} (\xi_{ij}-\eta_{ij}) \mathbf E_{\eta,x,y}^{n,k,t}\bigg[\int_0^t f_j(t-s) e_i(W_s)ds\bigg]\bigg| \\ &\leq \bigg[\sum_{i=1}^n\sum_{j=1}^{2k}(\xi_{ij}-\eta_{ij})^2\bigg]^{1/2}\bigg[\sum_{i=1}^n\sum_{j=1}^{2k}\bigg(\mathbf E_{\eta,x,y}^{n,k,t}\bigg[\int_0^t f_j(t-s) e_i(W_s)ds\bigg]\bigg)^2\bigg]^{1/2} \\ & = \|\xi^{n,k}-\eta^{n,k}\|_{L^2([0,T]\times[0,1])} \big(\mathbf E_{\eta,x,y}^{n,k,t}\big)^{\otimes 2} [\mathscr Q_{n,k}(t,W^1,W^2)]^{1/2}\\ &\leq \|\xi-\eta\|_{L^2([0,T]\times[0,1])} \cdot \lambda^{1/2}.
\end{align*}
This already gives the claim \eqref{dist} by combining with the previous expression. 

Now we are finally in a position to establish the claim \eqref{negmombo} stated at the beginning of this step. This will be done by combining \eqref{dist} and \eqref{lb} with the Gaussian isoperimetric inequality. The isoperimetric inequality states that if $A$ is a Borel measurable subset of any abstract Wiener space $(X,H,\mu)$ then for all $t>0$ one has that $\mu(\{\xi: dist_H(\xi,A) >t\}) \leq 1-\Phi\big(\Phi^{-1}(\mu(A))+t\big),$ where $\Phi$ is the cdf of a standard Gaussian \cite[Theorem 3.48]{Hai09}. Note that this bound is uniform over \textit{all} Borel sets $A$.

Take $\lambda = \lambda(\epsilon,T)$ as in \eqref{lambda}, and then based on this value of $\lambda$, choose some $a=a(\epsilon,T)\in \Bbb R$ so that the infimum in \eqref{lb} is precisely equal to $\Phi(a)$ where $\Phi$ is the cdf of a standard Gaussian. Finally, choose $\delta = \delta(\epsilon,T)>0$ so that $\delta < \inf_{x,y\in[0,1]} \inf_{t\in [\epsilon,T]} \frac12 p_t^{A,B}(x,y)$. Then for all $r>0$, we see by \eqref{dist} that 
\begin{align*}\sup_{n,k}\sup_{x,y\in[0,1]} \sup_{t\in[\epsilon,T]}\Bbb P(z_t^{n,k}(x,y) <r) &\leq  \sup_{n,k}\sup_{x,y\in[0,1]} \sup_{t\in[\epsilon,T]} \Bbb P\big(dist_{L^2}(\xi, A^{\lambda,n,k}_{t,x,y})> \lambda^{-1/2} \log(r/\delta)\big) \\ &\leq 1-\Phi\big(a+ \lambda^{-1/2} \log(r/\delta)\big),
\end{align*}
where we use \eqref{lb} and the isoperimetric inequality in the second bound. Now since $\Phi(x) \sim 1-e^{-x^2/2}$ for large $x$, this establishes a tail bound of order $e^{-c(\log r)^2}$. This is a stronger decay than any power and uniform over all $n,k$, and therefore by taking a limit (first letting $n\to\infty$ for fixed $k$, then letting $k\to \infty$ as described in Step 1) immediately gives \eqref{negmombo} as desired.
\\
\\
\textit{Step 3.} In this last step we need to justify how to turn the pointwise bound \eqref{negmombo} into a Holder bound as in the theorem statement. To do this, note that \begin{align}\label{rst}|z_t(x,a)^{-1}-z_t(x,b)^{-1}| &= z_t(x,a)^{-1}z_t(x,b)^{-1} |z_t(x,a)-z_t(x,b)|
\\|z_t(a,y)^{-1}-z_t(b,y)^{-1}|& = z_t(a,y)^{-1}z_t(b,y)^{-1} |z_t(a,y)-z_t(b,y)|\label{stu}
\\|z_t(x,y)^{-1}-z_s(x,y)^{-1}| &= z_t(x,y)^{-1}z_s(x,y)^{-1} |z_t(x,y)-z_s(x,y)|\label{tu}
\end{align}
for all $a,b,x,y,s,t.$ Consequently we have by Holder's inequality that 
\begin{align*}\Bbb E\big[|z_t(x,a)^{-1}-z_t(x,b)^{-1}|^p\big] &= \Bbb E\big[z_t(x,a)^{-p}z_t(x,b)^{-p} |z_t(x,a)-z_t(x,b)|^p\big] \\&\leq \Bbb E[ z_t(x,a)^{-3p}]^{1/3}\Bbb E[ z_t(x,b)^{-3p}]^{1/3} \Bbb E[ |z_t(x,a)-z_t(x,b)|^{3p}]^{1/3}. 
\end{align*}
Now for any $x,y,a,b\in [0,1]$ and any $t\in [\epsilon,T]$ we know by \eqref{negmombo} that $\Bbb E[ z_t(x,a)^{-3p}]^{1/3}\Bbb E[ z_t(x,b)^{-3p}]^{1/3}$ is bounded by some universal constant $C=C(\epsilon,T)$. by applying \eqref{bound2} in Theorem \ref{70}, the term $\Bbb E[ |z_t(x,a)-z_t(x,b)|^{3p}]^{1/3}$ can be bounded by $C|a-b|^{1/2-\delta}$ where $C$ is a universal constant. This gives us a bound on \eqref{rst} that will be conducive to applying Kolmogorov-Chentsov. To bound \eqref{stu}, note that for each fixed $t$ we have $(z_t(x,y))_{x,y\in[0,1]} \stackrel{d}{=} (z_t(y,x))_{x,y\in[0,1]}$ (by a time reversal of the noise) and so the same bound holds for \eqref{stu}. Finally, to bound \eqref{tu}, we can use a similar argument applying Holder's inequality in conjuction with \eqref{negmombo} and \eqref{bound2} to get a bound of the form $C|t-s|^{1/4-\delta}.$ Consequently by Kolmogorov-Chentsov (multiparameter version), we conclude the desired Holder continuity of $\mathbf {z}^{-1}$ appearing in the proposition statement. 

The desired bound on Holder continuity of $\mathbf{z}$ itself is similar, but one does not even need to consider the negative moments: the argument follows immediately from \eqref{bound2} together with $(z_t(x,y))_{x,y\in[0,1]} \stackrel{d}{=} (z_t(y,x))_{x,y\in[0,1]}$.
\end{proof}

For a compact metric space $X$ we define $C(X)$ to be the Banach space of all continuous functions $X\to \Bbb R.$

\begin{cor}\label{existence}
Let $\mathcal X$ be the state space defined in \eqref{xspace}. Define the Markov operator $\mathfrak P: C(\mathcal X) \to C(\mathcal X)$ as follows: for $F\in C(\mathcal X)$ and $\phi\in \mathcal X$ let\begin{equation}\label{markop}\mathfrak PF(\phi):= \Bbb E_\phi [ F([z(1,\cdot)]_\mathcal X)]\end{equation} where $z$ solves \eqref{mshe} and $[\mu]_\mathcal X$ denotes the equivalence class of the nonzero measure $\mu$, and the expectation on the right side is taken with respect to the process $z$ started at any representative of the equivalence class $\phi$. Then $\mathfrak P$ is globally well-defined on the quotient space $\mathcal X$, and has at least one invariant measure. Furthermore any invariant measure is supported on the space $\mathcal Y$ given by \eqref{yspace}.
\end{cor}

\begin{proof}
Proposition \ref{nega} implies that the solution $z$ of \eqref{mshe} does not collapse to zero in finite time. Thus $z(1,\cdot)$ is indeed a nonzero and non-negative function on $[0,1]$, consequently it does indeed yield a nonzero random measure on $[0,1]$ by identifying it with the measure $A\mapsto \int_A z(1,x)dx$. The fact that $\mathfrak PF$ is well-defined on equivalence classes of $\mathcal X$ is immediate from the fact that $cz$ solves \eqref{mshe} for all $c>0$, whenever $z$ solves \eqref{mshe}, e.g. by uniqueness of mild solutions stated in Theorem \ref{70}.

Note that $\mathcal X$ is compact, and therefore any collection of probability measures on $\mathcal X$ is tight. Thus existence follows immediately from the Krylov-Bogoliubov criterion, since we have already shown global well-posedness of the Markov operator of $[z(t,\cdot)]$ on $\mathcal X$. 

The fact that any invariant measure is supported on $\mathcal Y$ follows from the positivity of the solution (see Proposition \ref{nega}) together with the estimate \eqref{bound2} in Theorem \ref{70}. More precisely, positivity and \eqref{bound2} imply that if we start from any deterministic initial data $[\mu]\in \mathcal X$, then the time-one solution $[z^\mu(1,\cdot)]$ will be supported on $\mathcal Y$ almost surely. In particular, if we start the process from an invariant measure, then the time-one solution will be in $\mathcal Y$ almost surely.
\end{proof}

\section{Support theorem and proof of Theorem \ref{mr}}

In this section we will prove Theorem \ref{mr}, which will be reformulated as Corollary \ref{mr2} below. In this section the strong Feller property will be used, so it is crucially important that the driving noise is space-time white noise (not colored in space). First let us prove full support of the invariant measure that was guaranteed to exist in Corollary \ref{existence}. 

\begin{defn}[Continuum directed random polymer measure]
    Fix $t>0$. For each realization of the space-time white noise $\xi$ on $[0,1]\times [0,t]$, and each finite nonnegative measure $\mu$ on $[0,t]$ let $\mathbf P^{\xi,\mu}_{X^{t,x}}$ be the unique path measure on $C[0,t]$ such that for all $0\leq t_1<...<t_k\le t$ and $x_1,...,x_k\in [0,1]$ one has $$\mathbf P^{\xi,\mu}_{X^{t,x}} \big( X_{t_1}\in dx_1,..., X_{t_k}\in dx_k\big) = \frac{\int_{[0,1]} z_{0,t-t_k}(x_{k+1},x_k)\cdots z_{t-t_2,t-t_{1}}(x_2,x_1) z_{t-t_1,t}(x_1,x) \mu(dx_{k+1})}{\int_{[0,1]} z_{0,t}(x_{k+1},x) \mu(dx_{k+1})}dx_1\cdots dx_k.$$
    Here $(X_s)_{s\in [0,t]}$ denotes the canonical process on $C[0,t]$, and $z_{s,t}(x,\bullet)$ for $t>s$ denotes the solution at time $t$ of \eqref{mshe} started from $\delta_x$ at time $s$, all coupled to the same realization of $\xi$. We call this measure the \textit{continuum directed random polymer measure} on $C[0,t]$, starting from $x$ at $t=0$ and finishing with with endpoint distribution $z_{0,t}\mu/\int z_{0,t}d\mu$.
\end{defn}

In other words, $\mathbf P^{\xi,\mu}_{X^{t,x}}$ is the measure induced by leveraging the convolution property of the propagators $z_{s,t}(x,y)$ to specify the finite dimensional distributions 
for each individual realization of $\xi$, as done in \cite[Definition 4.1]{AKQ11}. We have introduced an extra time reversal because it will be useful for our purposes. The fact that such a measure is indeed supported on $C[0,t]$ is non-obvious and was shown in Section 5 of \cite{AKQ11}. On a purely formal level one has the formal Radon-Nikodym derivative with respect to the law $\mathbf P_{W^{t,x}}$ of a Brownian motion of diffusion rate 2 on $[0,t]$ started at $x$ and reflected at the boundaries of $[0,1]$: 
\begin{equation}\label{density}\frac{d\mathbf P^{\xi,\mu}_{X^{t,x}}}{d\mathbf P_{W^{t,x}}} ( W) \propto \mu(W(t))e^{-A L_0^W(t) - B L_1^W(t) +\int_0^t \xi(t-s,W(s))ds},
\end{equation}
where $L_a^W$ are the local times at the boundaries and $A,B$ are the Robin boundary parameters.

\begin{lem}\label{12}
Fix a deterministic continuous function $h:[0,T]\times [0,1]\to \Bbb R$, and consider the mild solution of the SPDE given by \begin{equation}\partial_t z = \partial_x^2 z + zh+ z\xi,\;\;\;\;
\partial_x z(t,0) =Az(t,0), \;\;\;\; \partial_x z(t,1) = B z(t,1),\;\;\; z(0,dx) = \mu(dx)\label{c}
\end{equation}
Then for a.e $\xi \in E$ the solution $z^h$ can be decomposed as $z^h(t,x) = z^0(t,x)v^h(t,x)$, where $z^0$ is just the mild solution of \eqref{c} with $h=0$ (i.e., the solution of \eqref{mshe}), and 
\begin{equation}\label{poly}v^h(t,x) = \mathbf E^{\xi,\mu}_{X^{t,x}} \bigg[ \exp\bigg(\int_0^t h(t-s,X(s))ds\bigg)\bigg],
\end{equation}
where the expectation is taken with respect to the \textit{continuum directed random polymer measure} $\mathbf P^{\xi,\mu}_{X^{t,x}}$ on $C[0,t]$.
\end{lem}

Formally at the level of Feynman-Kac, the decomposition claimed in Lemma \ref{12} is obvious by \eqref{density} and some simple algebraic manipulation, but because of the roughness of $\xi$ one needs some work to give a rigorous proof.

\begin{proof}
First of all, we note that the mild solution $z^h$ can be shown to exist using the Cameron-Martin theorem together with Theorem \ref{70} (the $h=0$ case). To rigorously establish the decomposition \eqref{poly} one first proves that it is true when $\xi$ is replaced by a noise $\xi^\gamma$ which is smooth in space and white in time, in which case we have an associated decomposition $z=z^{0,\gamma} \cdot v^h_\gamma$ for the solution of \eqref{c} where $z^{0,\gamma}$ solves \eqref{c} with $h=0$ and $\xi$ replaced by $\xi^\gamma$, and $v^h_\gamma$ is the same as in \eqref{poly} except the path measure $\mathbf P^{\xi,\mu}_{X^{t,x}}$ is replaced by the appropriate measure $\mathbf P^{\xi,\mu}_{X^{t,x}_\gamma}$ associated to $z^{0,\gamma}$.

In this case the claimed decomposition is straightforward to show from applying Feynman-Kac, because by Proposition \ref{mollif} the density formula \eqref{density} is actually rigorous modulo two caveats. Firstly $\mu(W(t))$ should be interpreted in terms of a disintegration of Brownian bridge measures, so that the endpoint distribution of the Brownian path is proportional to $p_t^{neu}(x,y)\mu(dx)$ as in Proposition \ref{mollif}. Secondly $\xi(t-s,W(s))$ needs to be replaced by $\xi^\gamma(t-s,W(s))-C_\gamma$ for a large constant $C_\gamma$ as described in Proposition \ref{mollif}. 
Then as $\gamma\to 0$ the solutions $z^{0,\gamma}(t,x)$ converge in law (in probability in the topology of $C([\epsilon,T]\times [0,1])$ for all $T>\epsilon>0$) to $z^0(t,x)$ as stated in Proposition \ref{mollif}. 

Since $h$ is bounded and continuous, the path functionals $G(X):=\exp\big(\int_0^t h(t-s,X(s))ds\big)$ are bounded and continuous on $C[0,1]$ so it just remains to show that for a.e. fixed realization of $\xi$, the path measures $\mathbf P^{\xi,\mu}_{X^{t,x}_\gamma}$ also converge in law as $\gamma\to 0$, as measures on $C[0,1]$ to the path measure $\mathbf P^{\xi,\mu}_{X^{t,x}}$. The finite-dimensional marginals certainly converge, simply by convergence of the $z^{0,\gamma}$ to $z^0$. Thus one only needs to establish a.s. tightness of these path measures on $C[0,1]$. For this one can use Fatou's lemma to note that for $\kappa<1/2$ one has $$\Bbb E\big[z^0(t,x)\liminf_{\gamma\to 0} \mathbf E^{\xi,\mu}_{X^{t,x}_\gamma}[ \|X\|_{C^{\kappa}[0,t]}]\big] \leq \liminf_{\gamma\to 0} \Bbb E\big[ z^{0,\gamma}(t,x) \mathbf E^{\xi,\mu}_{X^{t,x}_\gamma}[ \|X\|_{C^{\kappa}[0,t]}]\big] = \mathbf E_{BM} [ \|X\|_{C^{\kappa}[0,t]}]]<\infty.$$ Here the outer expectation $\Bbb E$ is over the noise $\xi$, and $\mathbf E_{BM}$ is the expectation with respect to a standard Brownian motion of rate 2. By strict positivity of $z_0(t,x)$, the above bound establishes $\xi$-a.s. finiteness of $\liminf_{\gamma\to 0} \mathbf E^{\xi,\mu}_{X^{t,x}_\gamma}[ \|X\|_{C^{\kappa}[0,t]}].$ In the last equality we are using that for each $\gamma>0$ one has $$\Bbb E\big[ z^{0,\gamma}(t,x) \mathbf E^{\xi,\mu}_{X^{t,x}_\gamma}[ \|X\|_{C^{\kappa}[0,t]}]\big] = \mathbf E_{BM} [ \|X\|_{C^{\kappa}[0,t]}],$$ which holds by \cite[Lemma 4.2]{AKQ11}. That result is for $\gamma=0$ and full space, but the result for the Robin case and $\gamma>0$ is completely analogous. This proves the desired a.s. tightness (at least along some random subsequence), thus establishing the validity of the a.s. decomposition claimed above, with $v^h$ given by \eqref{poly}.
\end{proof}

With the above preliminaries established, we can now prove our support theorem. 

\begin{thm}[Support theorem]\label{supp}
Fix any deterministic initial condition $[\mu]\in \mathcal X$. Let $z(t,x)$ denote the Itô solution of \eqref{mshe} started from $\mu$. Then the $\mathcal X$-valued random variable $[z(1,\cdot)]$ has full support in $\mathcal X$. In particular, any invariant measure for the $\mathcal X$-valued Markov process $[z(t,\cdot)]$ necessarily has full support in $\mathcal X.$
\end{thm}

\begin{proof}We break the proof into three steps. In this proof, we will set the terminal time $T=1$ and we will consider $\xi$ as living in some Banach space $E$ which can be explicitly taken as the closure of smooth functions with respect to the norm of some parabolic Hölder space $\mathcal C^{-3/2-\kappa}_{\mathfrak s}. $ The exact choice of space is not important here but what is important is that there is a continuously embedded Hilbert space $H=L^2([0,1]\times [0,T])$ (the Cameron-Martin space) in $E$ with the property that if $T_h(\xi) = \xi+h$ then the pushforward $T_h^*\mu$ of the law $\mu$ of $\xi$ on $E$ is absolutely continuous (in fact equivalent) with respect to $\mu$ and one has $\frac{d(T_h^*\mu)}{d\mu}(\xi) = e^{\langle \xi,h\rangle -\frac12\|h\|_H^2},$ where $\langle \cdot , h\rangle$ is the stochastic integral against $\xi$.
\\
\\
\textit{Step 1.} We claim that the law of any $z$ solving \eqref{mshe} under $T_h^*\mu$ equals that of the solution of \begin{equation}\partial_t z = \partial_x^2 z + zh+ z\xi,\;\;\;\;
\partial_x z(t,0) =Az(t,0), \;\;\;\; \partial_x z(t,1) = B z(t,1),\;\;\; z(0,dx) = \mu(dx)\label{c'}
\end{equation}
under the original measure with the same initial condition. Here $zh$ is an un-renormalized or ``classical'' product.

To prove this, note by the Cameron-Martin theorem that under $T_h^*\mu$ the random field $\eta:=\xi-h$ is a standard space-time white noise. Write the equation in mild form as in \eqref{mild}, and note that the integral against $z(s,y) \xi(ds,dy)$ can be written as an integral against $z(s,y)\big[ \eta(ds,dy) +h(s,y)dsdy\big]$, and we immediately obtain the claim.
\\
\\
\textit{Step 2. } Assume for this step that $(E,H,\gamma)$ is any abstract Wiener space, and $\mathcal X$ is another separable metric space. Let $\xi$ be sampled from $\gamma$, let $F:E\to \mathcal X$ be any measurable function, and let $D\subset H$ be any countable subset. Then we claim that the random set $F(\xi+D):=\{F(\xi+h):h\in D\}$ is almost surely contained in the support of $F(\xi)$. 

To prove this, let $U \subset \mathcal X$ be the complement of the support of $F(\xi)$. Then it is clear from the Cameron-Martin theorem that for every $h\in H$ one has $$\gamma(\{\xi: F(\xi+h)\in U\}) = 0.$$ Thus by countability of $D$ we have that $$\gamma(\{\xi: F(\xi+h)\in U\text{ for some } h\in D\}) \leq \sum_{h\in D} \gamma(\{\xi: F(\xi+h)\in U\}) = 0,$$which proves the claim.

This step will be used in the following form below: if there exists a deterministic countable subset $D\subset H$ such that the random set $F(\xi+D):=\{F(\xi+h):h\in D\}$ is almost surely dense in $\mathcal X$, then $F(\xi)$ necessarily has full support in $\mathcal X$.
\\
\\
\textit{Step 3.} We now let $F:E\to \mathcal X$ be the functional $F(\xi):= [z(1,\cdot)],$ where now $\mathcal X$ is the space given by \eqref{xspace} and $z$ solves \eqref{mshe} driven by $\xi$ starting from some arbitrary initial state $[\mu]\in\mathcal X$ (here the ``$\cdot$'' is in the spatial variable). In this case let us fix some smooth function $\varphi_0: [0,1]\to \Bbb R$, and then define $h_n^{\varphi_0}\in H=L^2([0,1]\times [0,T])$ by setting $$h_n^{\varphi_0}(t,x):=n 1_{[1-n^{-1},1]}(t)\varphi_0(x).$$ In this case by \eqref{c'} and Lemma \ref{12}, the random variable $F(\xi+h_n^{\varphi_0})$ is given by $[z_0(1,x)v_n(1,x)]$ where $z_0$ solves \eqref{mshe} and 
$$v_n(1,x) = \mathbf E^{\xi,\mu}_{X^{1,x}} \bigg[ \exp\bigg(n\int_0^{1/n} \varphi_0(X(s))ds\bigg)\bigg].$$ Now, since $X^{1,x}(0)=x$ and $X^{1,x}$ is a continuous path, as $n\to \infty$ the expression inside the expectation converges to $\exp(\varphi_0(x)).$ Moreover the expression inside the expectation is deterministically bounded above by $\exp(\|\varphi_0\|_\infty),$ and consequently $$\sup_{n\in\Bbb N}\sup_{x\in[0,1]} v_n(1,x) \leq \exp(\|\varphi_0\|_\infty).$$By the bounded convergence theorem we find that $v_n(1,x)$ converges a.s. to $e^{\varphi_0(x)},$ for each fixed $x\in [0,1]$. Likewise, the same bound may be used to conclude that $\int_{[0,1]} v_n(1,x)\psi(x) dx \stackrel{a.s.}{\to} \int_{[0,1]} e^{\varphi_0(x)}\psi(x)dx$ for all $\psi\in C[0,1]$.

Now we recall that $\mathcal X$ is a compact space, and therefore \textbf{any} collection of random variables in $\mathcal X$ is tight, so from the convergences stated in the previous paragraph, we can actually also conclude that $F(\xi+h_n^{\varphi_0})=[z_0(1,x)v_n(1,x)]$ converges in probability (with respect to the topology of $\mathcal X$) as $n\to\infty$ to $[z_0(1,x)e^{\varphi_0(x)}].$ Note here that $z_0$ has no dependence on $\varphi_0$ since only $v$ depends on $h$ in the decomposition from Step 2. 

Letting $D_0\subset C[0,1]$ be some countable dense subset consisting of smooth functions, we define $D:= \{ h^\varphi_n\in H: n\in \Bbb N, \varphi\in D_0\}$. We have shown that the closure in $\mathcal X$ of the random set $F(\xi+D)$ almost surely contains all elements of the form $[z(1,\cdot)e^\varphi]$ with $\varphi\in D_0$. Since $z(1,\cdot)$ is just a fixed positive function with no dependence on $\varphi\in D_0$, letting $\varphi$ vary throughout $D_0$ certainly gives a dense subset of $\mathcal X$, and therefore we conclude the proof of the theorem using the result of Step 3.
\end{proof}

Although Theorem \ref{supp} was proved using an argument specific to the KPZ equation, we remark that more abstract techniques such as those of \cite{HS19} could be used to prove a general Stroock-Varadhan type support theorem in much greater generality, in particular characterizing the support for the entire temporal trajectory (as opposed to just fixed time as we have done above) and in arbitrarily strong topologies. 

The above theorem is the only input to the uniqueness result that is specific to KPZ which is needed in order to prove the uniqueness theorem. The remainder of this section will focus on general results about strong Feller processes with invariant measures of full support, which will be useful in concluding the proof of the uniqueness theorem.

Recall that a Markov operator $\mathfrak P$ on a Polish space $\mathcal X$ is called \textbf{strong Feller} if it maps all bounded measurable functions to bounded continuous functions. Going forward, we will work extensively with this property. The next proposition allows us to determine when the strong Feller property can be ``upgraded'' from a smaller space with a finer topology to a larger space with a weaker topology.

\begin{prop}\label{upgrade}
Let $\mathfrak P$ be a Markov operator on some complete metric space $\mathcal X$. Let $\mathcal Y$ be another complete metric space contained densely in $\mathcal X$ (but with a separate metric) such that the inclusion map $\mathcal Y \hookrightarrow \mathcal X$ is continuous and such that open sets in $\mathcal Y$ are Borel measurable in $\mathcal X$. Suppose that the following assumptions hold: 

\begin{itemize}
    \item $\mathfrak P(1_\mathcal Y)=1$ identically, and the Markov operator $\bar{\mathfrak P}$ on $\mathcal Y$ defined by $\bar{\mathfrak P} f = \mathfrak P(f1_\mathcal Y)|_\mathcal Y$ is strong Feller. 

    \item The family of probability measures $\{\nu_\phi\}_{\phi\in \mathcal X}$ on $\mathcal Y$ defined by $\nu_\phi(A):=\mathfrak P1_A(\phi)$, satisfies the following regularity condition: whenever $\phi_n \stackrel{\mathcal X}{\longrightarrow} \phi$ the we have the weak convergence $\nu_{\phi_n} \to \nu_\phi$ as measures on $\mathcal Y.$
    
\end{itemize}
Then $\mathfrak P^2$ is a strong Feller operator on the larger space $\mathcal X$.
\end{prop}

We remark that $\nu_\phi=\mathfrak P_*\delta_\phi$ simply equals the law of the associated Markov process at time 1, when started from initial state $\phi$. The condition $\mathfrak P(1_\mathcal Y) = 1$ identically is the same as saying that the $\mathcal X$-valued Markov process associated to $\mathfrak P$ takes values in $\mathcal Y$ almost surely, regardless of the initial state. Then $\bar{\mathfrak P}$ is the Markov operator of the process viewed in the state space $\mathcal Y$. We will apply the theorem when $\mathcal X$ and $\mathcal Y$ are given by \eqref{xspace} and \eqref{yspace} respectively. 
Using a result of \cite{KM22}, this will then allow us to extend the strong Feller property to the the entire space $\mathcal X$ in Corollary \ref{SFP} knowing that it holds on $\mathcal Y$. 

\begin{proof}
Suppose $\phi_n \stackrel{\mathcal X}{\longrightarrow} \phi.$ By hypothesis $\int_\mathcal Y g d\nu_{\phi_n} \to \int_\mathcal Y g d\nu_\phi$ for all bounded continuous $g:\mathcal Y\to \Bbb R.$

Now since $\mathfrak P(1_\mathcal Y)=1$, for any bounded measurable $f:\mathcal X\to \Bbb R$ and any $\phi\in \mathcal X$ we have the identity $$\mathfrak P^2 f(\phi) = \int_\mathcal Y \bar{\mathfrak P}(f1_\mathcal Y) d\nu_\phi .$$ By the strong Feller property of $\bar{\mathfrak P}$ the integrand is a bounded continuous function on $\mathcal Y$, and therefore we conclude that if $\phi_n \stackrel{\mathcal X}{\longrightarrow} \phi$ then $$\lim_{n\to \infty} \mathfrak P^2 f(\phi_{n}) = \lim_{n\to \infty}\int_\mathcal Y \bar{\mathfrak P}(f1_\mathcal Y) d\nu_{\phi_n} = \int_\mathcal Y \bar{\mathfrak P}(f1_\mathcal Y) d\nu_{\phi} = \mathfrak P^2 f(\phi),$$ proving the claim.
\end{proof}

\begin{cor}\label{SFP}
The process $[z(t,\cdot)]$ where $z$ solves \eqref{mshe} is strong Feller on the space $\mathcal X$ given by \eqref{xspace}.
\end{cor}

\begin{proof}
In \cite[Theorem 1.1]{KM22} the strong Feller property was proved on the finer space \eqref{yspace}. Now thanks to the convergence statement in \eqref{conv}, we can immediately apply Proposition \ref{upgrade} to ``upgrade'' the strong Feller property to the coarser space $\mathcal X$, since we know it holds on $\mathcal Y$. Note here that we are implicitly using the fact that if $f_n,f$ are strictly positive and continuous random functions, and if $\|f_n-f\|_{C^\kappa[0,1]}\to 0$ in probability, then $\|\log(f_n)-\log(f)\|_{C^\kappa[0,1]}\to 0$ in probability, as is needed in \eqref{yspace}.
\end{proof}

\begin{prop}\label{0.1}
Let $\mathfrak P$ be a strong Feller Markov operator on some connected Polish space $\mathcal X$. 
Let $\rho$ be an invariant probability measure for $\mathfrak P$ with full topological support on $\mathcal X$. If $(X_n)_n$ is the associated Markov process started from \textbf{any} deterministic initial condition in $\mathcal X$, and if $f:\mathcal X\to \Bbb R$ is bounded and measurable, then as $N\to \infty$
$$N^{-1}\sum_{j=1}^N f(X_j) \to \int_\mathcal X f\;d\rho,\;\;\;\; a.s..$$
Let $\mathcal Y$ be another Polish space contained in $\mathcal X$ (but with a separate topology) such that the inclusion map $\mathcal Y \hookrightarrow \mathcal X$ is continuous and such that open sets in $\mathcal Y$ are Borel measurable in $\mathcal X$. If $\mathfrak P(1_\mathcal Y) = 1$ identically, then $\rho(\mathcal Y)=1$ and moreover, the Markov semigroup $\bar{\mathfrak P}$ on $\mathcal Y$ defined by $\bar{\mathfrak P} f = \mathfrak P(f1_\mathcal Y)|_\mathcal Y$ is also strong Feller and $\rho$ is the unique invariant measure for $\bar{\mathfrak P}.$ 
\end{prop}

Before proving the proposition, we remark that the above convergence holds for every (not just $\mu$-almost every) initial condition, which is very important because we do \textbf{not} assume that $\mu$ has full support in the finer space $\mathcal Y$, only in the coarser space $\mathcal X$. The relevance of this proposition is that (under the strong Feller assumption) even if we can prove that the invariant measure of some Markov process (usually some SPDE) has full support in some relatively weak topology (that of $\mathcal X)$, then that automatically implies uniqueness of the invariant measure for the process in any stronger topology (that of $\mathcal Y)$ as long as the Markov process takes values in the stronger topological space and the open sets of the stronger topological space are Borel measurable in the weaker topology. This proposition therefore generalizes \cite[ Corollaries 3.8 and 3.9]{HM17}. Note however that our proposition does \textit{not} necessarily imply full support of the invariant measure on $\mathcal Y,$ and in general this can actually be false, but the \textit{uniqueness} in $\mathcal Y$ still holds nonetheless (regardless of whether or not the full support holds).

\begin{proof}
Let $\rho$ be the invariant measure for $\mathfrak P$ with full support. By invariance we see that $$\rho(\mathcal Y) = \int_\mathcal X \mathfrak P(1_\mathcal Y)\;d\rho = \int_\mathcal X1 \; d\rho = 1.$$
The fact that $\bar{\mathfrak P}$ is strong Feller is immediate from the fact that if $g:\mathcal Y\to \Bbb R$ is bounded and measurable, then $g1_\mathcal Y$ is bounded and measurable from $\mathcal X\to \Bbb R$ (thanks to the fact that Borel subsets of $\mathcal Y$ are also Borel subsets of $\mathcal X)$. Now $\mathfrak P(g1_\mathcal Y)$ is bounded and continuous on $\mathcal X$, and since the topology of $\mathcal Y$ is finer than that of $\mathcal X$ by assumption, the restriction $\mathfrak P(g1_\mathcal Y)|_\mathcal Y$ is automatically continuous from $\mathcal Y\to \Bbb R.$

The proof of Corollary 3.9 in \cite{HM17} shows (using connectedness of $\mathcal X$ and the full support property) that $\rho$ is necessarily ergodic thanks to the support assumption plus the strong Feller property. Thus if $(X_n)_n$ is the Markov process in $\mathcal X$ started from the measure $\rho$, then $\frac1N\sum_{n=1}^N f (X_n) \to \int f\;d\rho$ almost surely for all bounded measurable $f:\mathcal X\to \Bbb R$. Letting $\Bbb P_\phi$ denote the law (on the canonical space $\mathcal X^\Bbb N)$ of the process $(X_n)_n$ started from some deterministic point $\phi\in \mathcal X$, this means that for each such $f$ one has 
\begin{equation}\label{full}\int_{\mathcal X} \Bbb P_\phi\bigg(\frac1N\sum_{n=1}^N f(X_n)\to \int f\;d\rho \bigg)\rho(d\phi)=1.
\end{equation}
Fix some bounded measurable function $f:\mathcal X\to \Bbb R$, and let $E_f$ be the Borel set defined by
$$E_f:=\bigg\{\phi\in \mathcal X:\Bbb P_\phi\bigg(\frac1N\sum_{n=1}^N f(X_n)\to \int f\;d\rho \bigg)=1\bigg\},$$ so that $\rho(E_f)=1$ by \eqref{full}. Note by the strong Feller property that $\mathfrak P1_{E_f}$ is a continuous function from $\mathcal X\to \Bbb R.$ Moreover by invariance of $\rho$ we have that 
$$\int_\mathcal X \mathfrak P1_{E_f}\;d\rho = \rho(E_f) = 1,$$
and therefore $\mathfrak P1_{E_f}(\phi) = 1$ for $\rho$-a.e. $\phi\in \mathcal X$. Since $\rho$ has full support on $\mathcal X$, any set of full measure is necessarily dense in $\mathcal X$. Thus the continuous function $\mathfrak P1_{E_f}$ equals 1 on a dense set, and therefore is identically 1 on all of $\mathcal X$. This means that 
\begin{equation}\Bbb P_\phi\bigg(N^{-1}\sum_{n=1}^N f(X_{n+1}) \to \int f\;d\rho\bigg)=1\label{eq1}
\end{equation}for \textbf{every} $\phi \in \mathcal X$ (not just on a set of full $\rho$-measure), so that $E_f$ is actually the entire space $\mathcal X$. The fact that $\rho$ is the unique invariant measure for $\bar{\mathfrak P}$ on $\mathcal Y$ is then immediate from \eqref{eq1} and the fact that $g1_\mathcal Y$ is a bounded measurable function on $\mathcal X$ whenever $g$ is a bounded measurable function on $\mathcal Y.$
\end{proof}




\begin{cor}\label{mr2}
Let $\mathcal X$ and $\mathcal Y$ be the state spaces defined in \eqref{xspace} and \eqref{yspace}, respectively. For any $A,B\in \Bbb R$, consider the $\mathcal X$-valued Markov process $X_{A,B}(t):=[z(t,0)]$ where $z$ solves \eqref{mshe}. Then, regardless of whether we view the state space as $\mathcal X$ or as $\mathcal Y$, the process $X_{A,B}$ has a unique ergodic invariant measure $\rho_{A,B}$ which is supported on $\mathcal Y$. Furthermore starting from \textbf{any} deterministic initial condition $[\mu]\in \mathcal X$, and for any bounded Borel-measurable function $F:\mathcal X\to \Bbb R$, one has the a.s. convergence $$\lim_{N\to\infty}\frac1N\sum_{n=1}^N F(X_{A,B}(n)) = \int_\mathcal Y F\;d\rho_{A,B}.$$
\end{cor}

\begin{proof}
Existence of an invariant measure for all $A,B\in\Bbb R$ follows immediately by compactness of the state space $\mathcal X$ as remarked in Corollary \ref{existence}. The strong Feller property of the process on $\mathcal X$ was shown in Corollary \ref{SFP}. The fact that $\rho_{A,B}$ has full support in $\mathcal X$ follows from Theorem \ref{supp}. Therefore we can conclude the uniqueness and ergodic theorem on both $\mathcal X$ and $\mathcal Y$ by applying Proposition \ref{0.1}.
\end{proof}

Let $C_+[0,1]$ denote all positive continuous functions on $[0,1]$. We note that if $F:C_+[0,1]\to \Bbb R$ is any measurable functional, then it gives rise to a family of measurable functionals from $\mathcal X\to \Bbb R$ by setting for instance $\bar F_1([\mu]):= F(f/f(1))$, $\bar F_2([\mu])=F(f/\int_{[0,1]}f)$, $\bar F_3([\mu]) = F(f/\|f\|_{L^p[0,1]})$, etc, where all of these are understood to be zero unless the measure $\mu$ is absolutely continuous with continuous density $f$. Another important class of measurable functions on $\mathcal X$ are induced by measurable functions $G: C^{-1/2-\kappa}[0,1]\to \Bbb R$ in which case we may define $\bar G:\mathcal X \to \Bbb R$ given by $\bar G([\mu]) = G(\partial_x \log f),$ again understood to be zero unless $\mu$ is absolutely continuous with strictly positive density $f\in C^{1/2-\kappa}[0,1]$. Such functionals $\bar G$ are related to the so-called stochastic Burgers equation with inhomogeneous Dirichlet boundary conditions on $[0,1].$ Consequently there is a fairly rich class of observables for which the above ergodic theorem applies. 

Another corollary is the existence of an almost sure exponential growth rate independent of the initial data: \begin{cor}\label{lim}Starting from any initial data $\nu$, the limit $$\lim_{N\to \infty} \frac1N\log z(N,0)$$
exists and is constant almost surely. The constant equals $\Bbb E_{\Bbb P_{\rho_{A,B}} }[\log z(1,0)-\log z(0,0)]$ 
and is therefore independent of the initial data $\nu$. Here $\Bbb P_{\rho_{A,B}}$ denotes the law on $\mathcal X^\Bbb N$ of the process started from stationarity.
\end{cor}

\begin{proof}
Define the canonical shift map $\theta : \mathcal X^\Bbb N\to \mathcal X^\Bbb N$ by sending $\theta((\phi_j)_j) = (\phi_{j+1})_j.$ Let $\Bbb P_\phi$ denote the measure on $\mathcal X^\Bbb N$ corresponding to the Markov process with initial state $\phi$. An easy consequence of Corollary \ref{mr2} is that for any $\phi\in \mathcal X$ and any measurable $F\in L^2(\Bbb P_\phi)$ that depends on only a finite number of coordinates, one has the convergence $$\frac1N \sum_{n=1}^N F\theta^n \to \int_{\mathcal X^\Bbb N} F \; d\Bbb P_{\rho_{A,B}},\;\;\;\;\;\;\;\Bbb P_\phi\text{-}a.s..$$ Letting $F(([z_j])_j):= \log z_1(0)-\log z_0(0)$ now immediately gives the claim (note that $F$ is well defined on equivalence classes, since scaling $z_0$ by some constant also scales $z_1$ by the same constant $\Bbb P_\phi$-almost surely).
\end{proof}

Next let us address the question of total variation convergence.

\begin{defn}\label{ufq}
Let $\mathcal X$ be a Polish space. If $\mathfrak P$ is a Markov operator on $\mathcal X$ and $\phi\in\mathcal X$, let $\nu_\phi(A):=\mathfrak P1_A(\phi),$ in other words $\nu_\phi = \mathfrak P_*\delta_\phi.$ Then $\mathfrak P$ is called \textit{ultra Feller} if $\|\nu_{\phi_n}-\nu_\phi\|_{TV}\to 0$ whenever $\phi_n\stackrel{\mathcal X}{\to} \phi$.
\end{defn}

\begin{prop}\cite[Theorem 1.6.6]{Hai07}\label{UFP}
Let $\mathcal X$ be a Polish space and let $\mathfrak P$ be a strong Feller Markov operator on $\mathcal X$. Then $\mathfrak P^2$ is ultra Feller.
\end{prop}

If $\mathcal M_1(\mathcal X)$ denotes the space of probability measures on the Polish space $\mathcal X$, then we claim that the ultra Feller property is equivalent to the adjoint operator $\mathfrak P_*$ being continuous from the topology of weak convergence on $\mathcal M_1(\mathcal X)$ to the topology of total variation convergence on $\mathcal M_1(\mathcal X)$. The backward implication is obvious. To prove the forward implication suppose $\gamma_n \in \mathcal M_1(\mathcal X)$ converges weakly to $\gamma \in \mathcal M_1(\mathcal X)$. Then by Skorohod's representation theorem we can choose $\mathcal X$-valued random variables $\boldsymbol{\phi}_n,\boldsymbol{\phi}$ such that the law of each $\boldsymbol{\phi}_n$ is $\gamma_n$, the law of $\boldsymbol{\phi}$ is $\gamma$, and $\boldsymbol{\phi}_n\to \boldsymbol{\phi}$ almost surely in the topology of $\mathcal X$. By assumption $\| \mathfrak P_*(\delta_{\boldsymbol{\phi}_n} - \delta_{\boldsymbol{\phi}})\|_{TV}\to 0$ almost surely. By the bounded convergence theorem this implies that $\Bbb E[\| \mathfrak P_*(\delta_{\boldsymbol{\phi}_n} - \delta_{\boldsymbol{\phi}})\|_{TV} ]\to 0$ which is strong enough to prove the claim since for each $n\ge 1$ we have $$\|\mathfrak P_*(\gamma_n-\gamma)\|_{TV}= \big\|\Bbb E[\mathfrak P_*(\delta_{\boldsymbol{\phi}_n} - \delta_{\boldsymbol{\phi}})] \big\|_{TV}\leq \Bbb E[\| \mathfrak P_*(\delta_{\boldsymbol{\phi}_n} - \delta_{\boldsymbol{\phi}})\|_{TV} ].$$

\begin{thm}[An automatic spectral gap from unique ergodicity]\label{tv3}
Suppose that $(\mathfrak P_t)_{t\ge 0}$ is a strongly continuous semigroup of Markov operators on a compact metric space $\mathcal X$. Assume that $\mathfrak P_t$ is strong Feller for each $t> 0$ and admits a unique invariant probability measure $\rho$. Then there exist $C,c>0$ such that for all $t\ge 0$ 
$$ \sup_{\phi\in\mathcal X} \big\| \rho - \mathfrak P_t^*\delta_\phi\big\|_{TV}\leq Ce^{-ct},$$ 
where $\mathfrak P_t^*$ is the adjoint Markov semigroup, and $\delta_\phi$ denotes a Dirac mass at the initial condition $\phi\in \mathcal X$. In particular this is true for the Markov operator of the open KPZ equation on 
the space $\mathcal X$ given in \eqref{xspace}.
\end{thm}

In other words, in continuous time, compactness plus the strong Feller property automatically allows us to upgrade unique ergodicity to geometric ergodicity (in fact a spectral gap). The intuitive idea of the proof is that compactness of $\mathcal X$ plus the strong Feller property makes the adjoint semigroup exhibit properties that are similar to those of a finite-state Markov chain, in particular an automatically positive spectral gap.

The reason that we are breaking our usual convention and using a continuous time parameter here is because we need it for the proof. The discrete time analogue is simply false: consider the simple random walk on $\Bbb Z/4\Bbb Z$ or even $\Bbb Z/2\Bbb Z$. But also notice how introducing exponential holding times immediately fixes the problem in either case. It might be possible to formulate some other reasonable statement in discrete time but it seems more difficult.

\begin{proof}We define $\mathcal M_0(\mathcal X)$ to be the Banach space consisting of all signed Borel measures $\gamma$ on $\mathcal X$ with $\gamma(\mathcal X)=0$, equipped with total variation norm. For a linear operator $T:\mathcal M_0(\mathcal X)\to \mathcal M_0(\mathcal X),$ define $\|T\|_{\text{op}}:=\sup_{\|\gamma\|_{TV}=1} \|T(\gamma)\|_{TV}$ to be the operator norm. Note that $\mathfrak P_t^*$ is a bounded linear operator from $\mathcal M_0(\mathcal X)\to \mathcal M_0(\mathcal X)$ with $\|\mathfrak P_t^*\|_{\text{op}} \leq 1.$

For all $\phi\in \mathcal X$ and $t>0$, we claim that $\frac1N\sum_{n=1}^N \mathfrak P_{nt}^*\delta_\phi$ converges weakly to $\rho.$ Indeed by compactness of $\mathcal X$, any subsequence is tight, hence has a further subsequence that converges weakly. Any such weak limit may be verified to be an invariant probability measure by directly applying $\mathfrak P_t^*-I$ and noting that the telescoping sums cancel. Hence any such weak limit must be equal to $\rho$. As every subsequence has a further subsequence converging to $\rho$, the original sequence must converge to $\rho$. 

From here it is clear that $\frac1N\sum_{n=1}^N \mathfrak P_{nt}^* \lambda$ converges weakly to $\lambda(\mathcal X)\rho$ for all nonnegative Borel measures $\lambda$ on $\mathcal X.$ Consequently for all $\gamma\in \mathcal M_0(\mathcal X)$, we have weak convergence of $N^{-1} \sum_{n=1}^N \mathfrak P_{nt}^*\gamma$ to the zero measure (use the Hahn-Jordan decomposition theorem to write $\gamma=\gamma_+-\gamma_-$ with $\gamma_+(\mathcal X) = \gamma_-(\mathcal X)$, and then apply $\frac1N\sum_{n=1}^N \mathfrak P_{nt}^*$ both components and take the limit).

We claim that $\mathfrak P_t^*$ is a compact operator on $\mathcal M_0(\mathcal X)$ for each $t> 0$. To prove this, first note that the sphere of radius 2 in $\mathcal M_0(\mathcal X)$ is compact with respect to the topology of weak convergence. This is true by Prohorov's theorem since each $\gamma\in\mathcal M_0(\mathcal X)$ with $\|\gamma\|_{TV}=2$ is a difference of probability measures, and because compactness of $\mathcal X$ ensures that any family of probability measures is automatically tight. The ultra Feller property guarantees that $\mathfrak P_t^*$ is continuous from the topology of weak convergence on probability measures to the topology of total variation convergence on probability measures, which proves compactness of the image under $\mathfrak P_t^*$ of the sphere of radius 2 in $\mathcal M_0(\mathcal X)$.

For this paragraph and the next, we will work with the complexification of $\mathfrak P_t^*$, which is an operator on the space $\mathcal M_0^\Bbb C(\mathcal X)$ of finite complex-valued Borel measures $\gamma$ on $\mathcal X$ with $\gamma(\mathcal X)=0$ (we will not distinguish the complexified operator from the real one). The norm on $\mathcal M_0^\Bbb C(\mathcal X)$ is defined by $\|\gamma\|_{TV}:=\sup\{|\gamma(f)| : f\in C(\mathcal X;\Bbb C)$ and $\sup|f|\leq 1\}$, where the bars now denote the complex modulus and $f$ can be a complex-valued function. Note that under this norm, the complexified operator still satisfies $\|\mathfrak P_t^*\|_{\text{op}}\leq 1$ with the same proof. The complexified operator is still compact since it can be viewed as the direct sum of two copies of the real operator. By the spectral theory of compact operators, the complex spectrum of $\mathfrak P_t^*$ consists of zero and some countable number of eigenvalues whose only accumulation point can be zero, and furthermore each of the eigenspaces must be finite dimensional. Since $\|\mathfrak P_t^*\|_{\text{op}}\leq 1$ all eigenvalues must have complex modulus less than or equal to 1. 

We claim that 1 cannot be an eigenvalue of $\mathfrak P_{t_0}^*$ for any $t_0> 0$. Indeed if $\mathfrak P_{t_0}^* \gamma = \gamma,$ then $\gamma = N^{-1}\sum_{n=1}^N \mathfrak P_{nt_0}^*\gamma$ for all $N$, and we already know that the latter converges weakly to the zero measure as $N\to\infty$ (that was in the case of a real-valued measure, but the complex case follows by applying the result individually to the real and imaginary parts of the measure). More generally we claim that $e^{i\theta}$ cannot be an eigenvalue of $\mathfrak P_{t_0}^*$ for any $\theta\in (0,2\pi]$ and any $t_0> 0$ (this is the step where we need a continuous time parameter). To obtain a contradiction, note that any eigenspace of $\mathfrak P_{t_0}$ is invariant under $\mathfrak P_t$ for all $t>0$. This is because $\mathfrak P_{t_0}$ and $\mathfrak P_t$ commute, so it follows that $\mathfrak P_t$ sends any eigenspace of $\mathfrak P_{t_0}$ to itself. Identify the $e^{i\theta}$-eigenspace of $\mathfrak P_{t_0}^*$ with $\Bbb C^{d}$, then we see that $\mathfrak P_{t}^*$ are linear maps from $\Bbb C^d\to \Bbb C^d$ that are continuous in $t\ge 0$ satisfying $\mathfrak P^*_{t+s}=\mathfrak P^*_t\mathfrak P^*_s$ for all $s,t\ge 0$ and $\mathfrak P^*_0=I_{d\times d}$, and $\mathfrak P_{t_0}^*=e^{i\theta} I_{d\times d}$. A $C_0$-semigroup on a finite-dimensional vector space has a generator, so we can write $\mathfrak P^*_t = e^{tA}$ for some complex $d\times d$ matrix $A$. Then $e^{(t_0A-i\theta I)}=I$, which implies that $t_0A-i\theta I$ must be conjugate to a diagonal matrix with diagonal entries $2\pi i k_1,...,2\pi i k_d$ for some $k_1,...,k_d \in \Bbb Z$ (by e.g. Jordan canonical form). In particular $A$ itself is conjugate to a diagonal matrix with entries $t_0^{-1}(2\pi k_1 +\theta)i,..., t_0^{-1}(2\pi k_d +\theta )i.$ In particular $\mathfrak P_t^*=e^{tA}$ has 1 as an eigenvalue whenever $t= 2\pi jt_0/(\theta + 2\pi k_1)$ for any $j\in\Bbb Z$ such that $2\pi jt_0/(\theta+2\pi k_1)>0$, which is a contradiction.

Consequently there are no eigenvalues of modulus 1, so the spectral radius of the compact operator $\mathfrak P_1^*$ is strictly less than 1. By Gelfand's spectral radius formula, this means that $\lim_{N\to\infty} \|\mathfrak P_N^*\|_{\text{op}}^{1/N}<1,$ or in other words $\|\mathfrak P_N^*\|_{\text{op}}\leq Ce^{-cN}$ for some $C,c>0$ independent of $N$. 
Since the invariant measure satisfies $\rho=\mathfrak P_N^*\rho$ for all $N$, and since $\delta_\phi-\rho\in\mathcal M_0(\mathcal X)$ for all $\phi\in\mathcal X$, this exponential bound on the operator norm then clearly gives $$\sup_{\phi\in\mathcal X}\big\| \mathfrak P_N^* \delta_\phi-\rho\big\|_{TV} =\sup_{\phi\in\mathcal X}\big\| \mathfrak P_N^* \big(\delta_\phi-\rho\big)\big\|_{TV}\leq \|\mathfrak P_N^*\|_{\text{op}}\cdot  \sup_{\phi\in\mathcal X} \|\delta_\phi-\rho\|_{TV}\leq Ce^{-cN}\cdot 2,$$ proving the claim.

For the special case of KPZ, Corollary \ref{SFP} shows that $\mathfrak P$ is strong Feller on $\mathcal X$. Moreover, Corollary \ref{mr2} clearly guarantees that unique ergodicity holds. Strong continuity of the semigroup follows from the Duhamel formula \eqref{mild} for the Hopf-Cole transform and noting that the stochastic part vanishes in the topology of $\mathcal X$ (for any initial condition) so that only the deterministic heat flow matters when taking the $t\to 0$ limit.
\end{proof}
We remark that $e^{i\theta}$ can indeed be an eigenvalue of $\mathfrak P_1^*$ on $\mathcal M_0(\mathcal X)$ in the discrete-time case: for example $e^{\pm 2\pi i/n}$ are eigenvalues in the case of the deterministic Markov chain on $\Bbb Z/n\Bbb Z$ where one always moves clockwise. We are uncertain if irrational multiples of $\pi$ are possible when the strong Feller property holds (this was the real obstruction to the formulation of a discrete time result).

\section{A coupling argument and proof of Theorem \ref{1f1s}}

Here we describe a different approach to proving uniqueness and geometric ergodicity of the stationary measure for open KPZ, which relies more heavily on the linear structure of the Hopf-Cole transform \eqref{mshe0}, and the interpretation through polymers.
A version of this method was implemented in \cite{GK20} for the periodic case. 
The methods described here still work to show uniqueness of the invariant measure in the case of driving noises that are colored in space, so long as they are still white in time (the only properties needed are negative moments of the solution and the convolution property).

\begin{defn}\label{cpk}
A continuous probability kernel (CPK) on $[0,1]$ is any function $\mathfrak p:[0,1]\times[0,1]\to \Bbb R$ that is continuous in both variables, strictly positive, and satisfies $\int_{[0,1]} \mathfrak p(x,y)dy=1$ for all $x\in [0,1]$.
\end{defn}

We remark that any two CPKs $\mathfrak p_1, \mathfrak p_2$ can be ``convolved'' or ``multiplied'' to form another CPK
$$(\mathfrak p_1\mathfrak p_2)(x,y):= \int_{[0,1]} \mathfrak p_1(x,z)\mathfrak p_2(z,y)dz.$$
This is an associative binary operation on CPKs, thus we may unambiguously write $\mathfrak p_1\cdots \mathfrak p_N.$ We have the following general result which essentially gives total variation bounds of exponential form for such products of CPKs. The main idea is that the Doeblin-type conditions implicit in the definition of the CPK can be bootstrapped to give nice bounds.

\begin{lem}\label{relaxation}
Let $\{\mathfrak p_n\}_{n\in\Bbb N}$ be any collection of CPKs, and define $\mathfrak q_N := \mathfrak p_1\cdots \mathfrak p_N$. Then for all $\delta>0$ one has $$\sup_{x,y\in[0,1]}\|\mathfrak q_N(x,\cdot)-\mathfrak q_N(y,\cdot)\|_{L^1[0,1]} \leq 2 (1-\delta)^{J_N(\delta)},$$ where $$J_N(\delta):= \#\{ k\leq N : \inf_{x,y\in [0,1]} \mathfrak p_k(x,y)>\delta \}.$$
\end{lem}

\begin{proof}
Let $A:=\{ k\in\Bbb N : \inf_{x,y\in [0,1]} \mathfrak p_k(x,y)>\delta \}. $ For $n\in A$ we define $$\tilde p_n(x,y):= \frac{\mathfrak p_n(x,y)-\delta}{1-\delta}.$$

Next we will construct a Markov chain $(S_n,T_n)_{n\ge 0}$ on $[0,1]^2$ whose marginals at time $N$ have law $\mathfrak q_N(x,\cdot)$ and $\mathfrak q_N(y,\cdot)$. Define $S_0=x$ and $T_0=y$. Inductively assuming $(S_k,T_k)_{k\leq n}$ have already been defined, we choose $(S_{n+1},T_{n+1})$ according to four different cases as follows:
\begin{itemize}
\item if $S_n=T_n$ then sample $S_{n+1}=T_{n+1}$ according to $\mathfrak p_n(S_n,\cdot).$ 
\item if $n\notin A$ and $S_n\neq T_n$ then independently sample $S_{n+1}$ and $T_{n+1}$ from $\mathfrak p_n(S_n,\cdot)$ and $\mathfrak p_n(T_n,\cdot)$ respectively.
\item if $n\in A$ and $S_n\neq T_n$, flip an independent coin that is heads with probability $\delta$:

\subitem $\bullet$ if heads, then sample $S_{n+1}=T_{n+1}$ \textit{uniformly} at random from $[0,1]$. 
\subitem $\bullet$ if tails, then independently sample $S_{n+1}$ and $T_{n+1}$ according to $\tilde p_n(S_n,\cdot)$ and $\tilde p_n(T_n,\cdot)$ respectively.
\end{itemize}
One may check that both $S_n$ and $T_n$ are Markov chains whose transition kernel at step $n$ is precisely $\mathfrak p_n$, and consequently $S_n$ has law $\mathfrak q_N(x,\cdot)$ and $T_n$ has law $\mathfrak q_N(y,\cdot)$. Furthermore the trajectories of $S,T$ coincide as soon as the first heads appears. The probability of exceeding some deterministic time $N$ before reaching a heads is equal to $(1-\delta)^{J_N}.$ 

From this coupling, we have shown that $$\sup_{x,y}\|\mathfrak q_N(x,\cdot)-\mathfrak q_N(y,\cdot)\|_{TV}\leq 2(1-\delta)^{J_N}.$$ For absolutely continuous measures, the total variation is just the $L^1$ norm, i.e., $$\|\mathfrak q_N(x,\cdot)-\mathfrak q_N(y,\cdot)\|_{TV} = \|\mathfrak q_N(x,\cdot)-\mathfrak q_N(y,\cdot)\|_{L^1[0,1]}.$$
\end{proof}

As an interesting corollary of the previous lemma, note that if $\mathfrak p_n$ are \textit{random} CPKs which are chosen IID from some probability distribution on the space of CPKs, then clearly $J_N(\delta)$ is distributed as Binomial$(N,\epsilon)$ for some choice of $\delta,\epsilon>0$ and therefore the above lemma shows the exponential convergence $$\Bbb E[\sup_{x,y} \|\mathfrak q_N(x,\cdot)-\mathfrak q_N(y,\cdot)\|_{L^1}] \leq Ce^{-cN}$$ for some constants $c,C$ depending on the distribution of $\mathfrak p_1$ but not $N$. Then using Borel-Cantelli and Markov's inequality, this expectation bound also implies almost-sure statements such as $$\sup_{x,y} \|\mathfrak q_N(x,\cdot)-\mathfrak q_N(y,\cdot)\|_{L^1}\to 0,\;\;\;\;\;a.s..$$ This will be the main idea going forward, however the IID case is very simple and we will actually need to apply the theorem in a more complicated situation.

\begin{thm}[One-force one-solution principle]\label{1f1s2}
Fix a space-time white noise on $(-\infty,0]\times[0,1]$ and let $z_N^\mu,z_N^\nu: [-N,0]\times [0,1]\to \Bbb R$ denote the solutions to open KPZ driven by $\xi$ started from any deterministic nonnegative Borel measures $\mu,\nu$ at time $t=-N$. Then 
there exist $C,c>0$ such that for all $N \ge 1$ one has
$$\Bbb E \bigg[ \sup_{[\mu],[\nu]\in\mathcal X} d_\mathcal X\big([z_N^\mu(0,\cdot)],[z_N^\nu(0,\cdot)]\big)\bigg] \leq Ce^{-cN}.$$ 
Here $\mathcal X$ is the space defined in \eqref{xspace}.
\end{thm}

\begin{proof}Fix a realization of $\xi$ on $(-\infty,0]\times [0,1]$, and let $z_{s,t}(x,y)$ denote the propagators of the stochastic heat equation for $s<t\leq 0$ and $x,y\in[0,1]$. Fix some positive continuous deterministic $f:[0,1]\to\Bbb R_+.$ Then define for $n\ge 0$ $$\mathfrak p_{n}^f(x,y):= \frac{z_{-n-1,-n}(x,y)\int_{[0,1]} z_{-n,0}(y,a)f(a)da}{\int_{[0,1]} z_{-n-1,0}(x,a)f(a)da}, $$ where $z_{0,0}$ is understood to be $\delta_0$ so the integral in the numerator for $n=0$ is understood to be $f(y)$. Note by Proposition \ref{convo} that $\mathfrak p_{n}^f$ is indeed a CPK. Then for $N\ge 1$ let $\mathfrak q_N^f:=\mathfrak p_{N-1}^f\cdots \mathfrak p_{0}^f,$ so that $$\mathfrak q_N^f(x,y) = \frac{z_{-N,0}(x,y)f(y)}{\int_{[0,1]}z_{-N,0}(x,a)f(a)da}.$$
For $\delta>0$ and $n\ge 0$ we define $\Omega_n$ to be the event 
\begin{align*}\label{event}\big\{\inf_{u,x,y\in[0,1]} &z_{-n-1,-n}(x,y)\wedge z_{-n,1-n}(y,u) \wedge z_{-n-1,1-n}(x,u) >\delta \big\}\\&\cap \big\{\sup_{u,x,y\in[0,1]} z_{-n-1,-n}(x,y)\vee z_{-n,1-n}(y,u) \vee z_{-n-1,1-n}(x,u) <\delta^{-1}\big\}
\end{align*}
From now onward, fix some $\delta>0$ (small enough) so that the event $\Omega_n$ has strictly positive probability $\epsilon>0$ (note that $\epsilon$ does not depend on $n$ by temporal stationarity of the noise $\xi$). Note foremost that $\Omega_n$ is measurable with respect to the noise $\xi$ restricted to $[-n-2,-n]\times[0,1]$. Consequently $\{\Omega_n\}_{n\in2\Bbb N}$ are independent events. Furthermore on $\Omega_n$ we have the bound 
\begin{align*}
    \mathfrak p_{n}^f(x,y) &= \frac{z_{-n-1,-n}(x,y)\int_{[0,1]}z_{-n,0}(y,a)f(a)da}{\int_{[0,1]}z_{-n-1,0}(x,a)da} \\ &> \frac{\delta\int_{[0,1]} z_{-n,0}(x,a)da}{\int_{[0,1]}z_{-n-1,0}(y,a)f(a)da} \\ &= \frac{\delta \int_{[0,1]}\big(\int_{[0,1]} z_{-n,1-n}(y,u)z_{1-n,0}(u,a)du\big)f(a)da}{\int_{[0,1]}\big(\int_{[0,1]} z_{-n-1,1-n}(x,u)z_{1-n,0}(u,a)du\big)f(a)da} \\ &=\frac{\delta \int_{[0,1]} z_{-n,1-n}(y,u)\big(\int_{[0,1]}z_{1-n,0}(u,a)f(a)da\big)du}{\int_{[0,1]} z_{-n-1,1-n}(x,u)\big(\int_{[0,1]}z_{1-n,0}(u,a)f(a)da\big)du}\\  &>\frac{\delta \int_{[0,1]} \delta\big(\int_{[0,1]}z_{1-n,0}(u,a)f(a)da\big)du}{\int_{[0,1]} \delta^{-1}\big(\int_{[0,1]}z_{1-n,0}(u,a)f(a)da\big)du} \\ &=\delta^3.
\end{align*}
Note that this bound holds \textit{uniformly} over all $x,y\in [0,1]$ and $f\in C_+[0,1]$. In the notation of Lemma \ref{relaxation} this means that $$J_{N}(\delta^3) \geq \sum_{n\leq N/2} 1_{\Omega_{2n}}. $$ Therefore by Lemma \ref{relaxation} we find that $$\sup_{\substack{x,\bar x\in [0,1]\\ f\in C_+[0,1]}}\|\mathfrak q_N^f(x,\cdot)-\mathfrak q_N^f(\bar x,\cdot)\|_{L^1[0,1]}<(1-\delta^3)^{J_N(\delta^3)}\leq (1-\delta^3)^{\sum_{n\leq \lfloor N/2\rfloor} 1_{\Omega_{2n}}}.$$ But the $\Omega_{2n}$ are independent, so the random variable $\sum_{n\leq \lfloor N/2\rfloor} 1_{\Omega_n}$ is distributed as Binomial$(\lfloor N/2 \rfloor,\epsilon).$ Thus there exist constants $C,c>0$ independent of $N$ such that $$\Bbb E[ (1-\delta^3)^{\sum_{n\leq \lfloor N/2\rfloor} 1_{\Omega_{2n}}}]<Ce^{-cN}.$$ 
So far this shows that $$\Bbb E\bigg[ \sup_{\substack{x,\bar x\in [0,1]\\ f\in C_+[0,1]}} \int_{[0,1]}\bigg| \frac{z_{-N,0}(x,y)}{\int_{[0,1]}z_{-N,0}(x,a)f(a)da}-\frac{z_{-N,0}(\bar x,y)}{\int_{[0,1]}z_{-N,0}(\bar x,a)f(a)da}\bigg|f(y)dy \bigg] < Ce^{-cN}.$$
Now notice that for any $f,g\in C_+[0,1]$ we can write $g=(g/f)\cdot f$ so that
\begin{align*}
    \sup_{x,\bar x\in [0,1]} \int_{[0,1]}&\bigg| \frac{z_{-N,0}(x,y)}{\int_{[0,1]}z_{-N,0}(x,a)f(a)da}-\frac{z_{-N,0}(\bar x,y)}{\int_{[0,1]}z_{-N,0}(\bar x,a)f(a)da}\bigg|g(y)dy \\ &\leq \|g/f\|_{C[0,1]}\sup_{x,\bar x\in [0,1]}\int_{[0,1]}\bigg| \frac{z_{-N,0}(x,y)}{\int_{[0,1]}z_{-N,0}(x,a)f(a)da}-\frac{z_{-N,0}(\bar x,y)}{\int_{[0,1]}z_{-N,0}(\bar x,a)f(a)da}\bigg|f(y)dy.
\end{align*}
This will be the key estimate going forward, and from now onwards it is mostly a matter of elementary algebraic manipulations. For $M>0$ (large) define $$\mathscr C_M:=\{(f,g)\in C_+[0,1]\times C_+[0,1] : \|g/f\|_{C[0,1]}+\|f/g\|_{C[0,1]}<M\}.$$ 
Then, because of the factor $\|g/f\|_{C[0,1]}$ appearing in the bound above, we see that $$\Bbb E\bigg [\sup_{\substack{x,\bar x\in [0,1]\\ (f,g)\in \mathscr C_M}} \int_{[0,1]}\bigg| \frac{z_{-N,0}(x,y)}{\int_{[0,1]}z_{-N,0}(x,a)f(a)da}-\frac{z_{-N,0}(\bar x,y)}{\int_{[0,1]}z_{-N,0}(\bar x,a)f(a)da}\bigg|g(y)dy\bigg] \leq CMe^{-cN}.$$
By a time reversal of the noise we have for each $N$ the distributional equality $(z_{-N,0}(x,y))_{x,y\in [0,1]} \stackrel{d}{=} (z_{-N,0}(y,x))_{x,y\in [0,1]}$, so the preceding expectation bound can be rewritten as $$\Bbb E\bigg [\sup_{\substack{x,\bar x\in [0,1]\\ (f,g)\in \mathscr C_M}} \int_{[0,1]}\bigg| \frac{z_{-N,0}(y,x)}{\int_{[0,1]}z_{-N,0}(a,x)f(a)da}-\frac{z_{-N,0}(y,\bar x)}{\int_{[0,1]}z_{-N,0}( a,\bar x)f(a)da}\bigg|g(y)dy\bigg] \leq CMe^{-cN}.$$
Let $z_{-N,0}^f(x):= \int_{[0,1]} z_{-N,0}(a,x)f(a)da$ which is the solution at time 0 of the stochastic heat equation started from initial condition $f$ at time $-N$. Notice that 
\begin{align*}
    \bigg|\frac{z_{-N,0}^g(x)}{z_{-N,0}^f(x)}-\frac{z_{-N,0}^g(\bar x)}{z_{-N,0}^f(\bar x)}\bigg| & = \bigg| \int_{[0,1]}\bigg( \frac{z_{-N,0}(y,x)}{\int_{[0,1]}z_{-N,0}(a,x)f(a)da}-\frac{z_{-N,0}(y,\bar x)}{\int_{[0,1]}z_{-N,0}( a,\bar x)f(a)da}\bigg)g(y)dy\bigg|\\&\leq \int_{[0,1]}\bigg| \frac{z_{-N,0}(y,x)}{\int_{[0,1]}z_{-N,0}(a,x)f(a)da}-\frac{z_{-N,0}(y,\bar x)}{\int_{[0,1]}z_{-N,0}( a,\bar x)f(a)da}\bigg|g(y)dy,
\end{align*}therefore we see from the previous bound that $$\Bbb E \bigg[ \sup_{\substack{x,\bar x\in [0,1]\\ (f,g)\in \mathscr C_M}} \bigg|\frac{z_{-N,0}^g(x)}{z_{-N,0}^f(x)}-\frac{z_{-N,0}^g(\bar x)}{z_{-N,0}^f(\bar x)}\bigg| \bigg] \leq C M e^{-cN}.$$ Setting $\bar x=0$ then implies that 
$$\Bbb E\bigg[ \sup_{(f,g)\in \mathscr C_M}\bigg\| \frac{z^g_{-N,0}}{z^f_{-N,0}} - \frac{z^g_{-N,0}(0)}{z^f_{-N,0}(0)}\bigg\|_{C[0,1]}\bigg] \leq CMe^{-cN}.$$
Next notice that 
$$\bigg| \frac{\int z^g_{-N,0}}{\int z^f_{-N,0}}-\frac{z^g_{-N,0}(0)}{z^f_{-N,0}(0)} \bigg| = \frac{\big| \int z^f_{-N,0} \big(\frac{z^g_{-N,0}}{z^f_{-N,0}} - \frac{z^g_{-N,0}(0)}{z^f_{-N,0}(0)}\big)\big| }{\int z^f_{-N,0}} \leq \bigg\| \frac{z^g_{-N,0}}{z^f_{-N,0}} - \frac{z^g_{-N,0}(0)}{z^f_{-N,0}(0)}\bigg\|_{C[0,1]}.$$
The last two convergences imply \begin{equation}\label{a1}\Bbb E \bigg[ \sup_{(f,g)\in \mathscr C_M} \bigg\| \frac{z^g_{-N,0}}{z^f_{-N,0}} - \frac{\int z^g_{-N,0}}{\int z^f_{-N,0}}\bigg\|_{C[0,1]}\bigg] \leq CMe^{-cN}.
\end{equation}
Note that $f$ and $g$ are symmetric in the definition of $\mathscr C_M$, and therefore we also have 
\begin{equation}\label{a2}\Bbb E \bigg[ \sup_{(f,g)\in \mathscr C_M}\bigg\| \frac{z^f_{-N,0}}{z^g_{-N,0}} - \frac{\int z^f_{-N,0}}{\int z^g_{-N,0}}\bigg\|_{C[0,1]}\bigg] \leq CMe^{-cN}.
\end{equation}
For all bounded measurable functions $\varphi:[0,1]\to\Bbb R$ one has that 
\begin{align*}\bigg| \frac{\int z^g_{-N,0}\varphi}{\int z^g_{-N,0}}-\frac{\int z^f_{-N,0}\varphi}{\int z^f_{-N,0}} \bigg| 
&= \bigg|\frac{ \int z^f_{-N,0} \varphi \big(\frac{z^g_{-N,0}}{z^f_{-N,0}} - \frac{\int z^g_{-N,0}}{\int z^f_{-N,0}}\big) }{\int z^g_{-N,0}} \bigg|
\\ &\leq \|\varphi\|_{L^\infty[0,1]}\frac{\int z_{-N,0}^f}{\int z^g_{-N,0}}\bigg\| \frac{z^g_{-N,0}}{z^f_{-N,0}} - \frac{\int z^g_{-N,0}}{\int z^f_{-N,0}}\bigg\|_{C[0,1]}.
\end{align*}
The left side is symmetric in $f$ and $g$, thus we also have 
$$\bigg| \frac{\int z^g_{-N,0}\varphi}{\int z^g_{-N,0}}-\frac{\int z^f_{-N,0}\varphi}{\int z^f_{-N,0}} \bigg|  \leq \|\varphi\|_{L^\infty[0,1]}\frac{\int z_{-N,0}^g}{\int z^f_{-N,0}}\bigg\| \frac{z^f_{-N,0}}{z^g_{-N,0}} - \frac{\int z^f_{-N,0}}{\int z^g_{-N,0}}\bigg\|_{C[0,1]}.$$
Note that $\min\big\{\frac{\int z_{-N,0}^g}{\int z^f_{-N,0}},\frac{\int z_{-N,0}^f}{\int z^g_{-N,0}}\big\}\leq 1 $, and therefore the last two expressions combined give 
$$\bigg| \frac{\int z^g_{-N,0}\varphi}{\int z^g_{-N,0}}-\frac{\int z^f_{-N,0}\varphi}{\int z^f_{-N,0}} \bigg|  \leq \|\varphi\|_{L^\infty[0,1]} \max \bigg\{\bigg\| \frac{z^g_{-N,0}}{z^f_{-N,0}} - \frac{ \int z^g_{-N,0}}{ \int z^f_{-N,0}}\bigg\|_{C[0,1]},\bigg\|\frac{ z_{-N,0}^f}{ z^g_{-N,0}} - \frac{\int z^f_{-N,0}}{\int z^g_{-N,0}}\bigg\|_{C[0,1]} \bigg\}. $$
Combining this with \eqref{a1} and \eqref{a2}, and taking a sup over $\|\varphi\|_{L^\infty}\leq 1$, gives us that $$\Bbb E \bigg[ \sup_{\substack{(f,g)\in \mathscr C_M\\ \|\varphi\|_{L^\infty[0,1]}\leq 1}} \bigg|\frac{\int z^g_{-N,0}\varphi}{\int z^g_{-N,0}}-\frac{\int z^f_{-N,0}\varphi}{\int z^f_{-N,0}}\bigg| \bigg] \leq CMe^{-cN}, $$ or in other words 
\begin{equation}\label{t2}\Bbb E \bigg[ \sup_{(f,g)\in \mathscr C_M} \bigg\| \frac{z^g_{-N,0}}{\int z^g_{-N,0}} - \frac{z^f_{-N,0}}{\int z^f_{-N,0}}\bigg\|_{L^1[0,1]}\bigg] \leq CMe^{-cN}.\end{equation}
Since the $L^1$ bounded below by that of $\mathcal X$ (recall we assumed in \eqref{xspace} that $d_{\text{Proh}}(\mu,\nu)\leq \|\mu-\nu\|_{TV}$), this actually implies
\begin{equation}\label{t1}
    \Bbb E \bigg[ \sup_{(f,g)\in \mathscr C_M} d_\mathcal X\big( [z^g_{-N,0}],[z^f_{-N,0}]\big)\bigg] \leq CMe^{-cN}.
\end{equation}

Now we are in a position to finally prove the theorem. Let $\mathcal M_1[0,1]$ denote the set of probability measures on $[0,1]$. Let $N$ be fixed henceforth (but note that all constants are still independent of $N$). Let $z_{-N,-n}^\mu = \int_{[0,1]} z_{-N,-n}(x,\cdot) \mu(dx) \in C_+[0,1]$ denote the solution at time $-n$ of the stochastic heat equation started from $\mu$ at initial time $-N$. For $M>0$ let $\Theta_M$ be the event $\{ (z_{-N,1-N}^\mu, z_{-N,1-N}^{\nu}) \in \mathscr C_M$ for all $\mu,\nu \in \mathcal M_1[0,1]\}$, and we let $\Gamma_M:=\Theta_M \backslash \Theta_{M-1}$. Notice by the Markov property of the stochastic heat equation (looking at the solution from time $N-1$ onwards) that 
$$1_{\Gamma_M}\cdot \sup_{[\mu],[\nu]\in\mathcal X} d_\mathcal X\big( [z^\mu_{-N,0}],[z^\nu_{-N,0}]\big) \leq 1_{\Gamma_M}\cdot \sup_{(f,g)\in\mathscr C_M} d_\mathcal X\big( [z^f_{1-N,0}],[z^g_{1-N,0}]\big).$$
Note that for all $M$ the random variable $1_{\Gamma_M}$ is measurable with respect to the noise on $[-N,1-N]$, whereas the random variable $\sup_{(f,g)\in\mathscr C_M} d_\mathcal X\big( [z^f_{1-N,0}],[z^g_{1-N,0}]\big)$ is measurable with respect to the noise on $[1-N,0]$, therefore the two are independent and so 
\begin{align}\notag\Bbb E \bigg[ 1_{\Gamma_M}\cdot \sup_{[\mu],[\nu]\in\mathcal X} d_\mathcal X\big( [z^\mu_{-N,0}],[z^\nu_{-N,0}]\big)\bigg] &\leq  \Bbb P(\Gamma_M) \Bbb E \bigg[ \sup_{(f,g)\in \mathscr C_M} d_\mathcal X\big( [z^g_{1-N,0}],[z^f_{1-N,0}]\big)\bigg]\\&\leq \Bbb P(\Gamma_M)\cdot CMe^{-c(N-1)},\label{ref1}
\end{align}
where we applied \eqref{t1}. Now we let $\Gamma_{\infty}$ be the complement of $\bigcup_M\Gamma_M = \bigcup_M \Theta_M$ and we obtain that
\begin{align*}\Bbb E\bigg [ \sup_{[\mu],[\nu]\in\mathcal X} d_\mathcal X\big( [z^\mu_{-N,0}],[z^\nu_{-N,0}]\big) \bigg] &= \Bbb E \bigg [ 1_{\bigcup_{M\in \Bbb N\cup\{\infty\}} \Gamma_M}\cdot  \sup_{[\mu],[\nu]\in\mathcal X} d_\mathcal X\big( [z^\mu_{-N,0}],[z^\nu_{-N,0}]\big)\bigg] \\ &\leq \sum_{M\in\Bbb N\cup\{\infty\}} \Bbb E\bigg[1_{ \Gamma_M}\cdot  \sup_{[\mu],[\nu]\in\mathcal X} d_\mathcal X\big( [z^\mu_{-N,0}],[z^\nu_{-N,0}]\big)\bigg]\\ &\stackrel{\eqref{ref1}}{\leq} Ce^{-cN} \sum_{M\in\Bbb N\cup\{\infty\}} M\Bbb P(\Gamma_M) \\ &\leq Ce^{-cN} \sum_{M\in\Bbb N} M\Bbb P(\Theta_{M-1}^c),
\end{align*}
where $\Theta_M^c$ denotes the complement. Therefore we just need to estimate the sum over $M$ (and in particular show that it is finite). To do that, we see that $$\sup_{\substack{ \mu\in \mathcal M_1[0,1]\\y\in[0,1]}} z^\mu_{-N,1-N}(y) =\sup_{\substack{ \mu\in \mathcal M_1[0,1]\\y\in[0,1]}} \int_{[0,1]} z_{-N,1-N}(x,y)\mu(dx) \leq \sup_{x,y\in[0,1]} z_{-N,1-N}(x,y),$$
$$\inf_{\substack{ \mu\in \mathcal M_1[0,1]\\x\in[0,1]}} z^\mu_{-N,1-N}(y) =\inf_{\substack{ \mu\in \mathcal M_1[0,1]\\y\in[0,1]}} \int_{[0,1]} z_{-N,1-N}(x,y)\mu(dx) \geq \inf_{x,y\in[0,1]} z_{-N,1-N}(x,y),$$ and therefore $$\Theta_M \subset \{ M^{-1/2} \leq \inf_{x,y\in[0,1]} z_{-N,1-N}(x,y) \leq \sup_{x,y\in[0,1]} z_{-N,1-N}(x,y) \leq M^{1/2}\}.$$
Therefore by taking complements and applying Markov's inequality with a sixth moment bound, we find that 
\begin{align*}
    \Bbb P(\Theta_M^c) \leq \frac{\Bbb E [\sup_{x,y\in[0,1]}z_{0,1}(x,y)^6]}{M^3}+\frac{\Bbb E\big[\big(\inf_{x,y\in[0,1]}z_{0,1}(x,y)\big)^{-6}\big]}{M^3}.
\end{align*}
where we used the fact that $z_{-N,1-N} \stackrel{d}{=}z_{0,1}.$ That the above moments are finite follows from Proposition \ref{nega}. This is enough to prove convergence of the above sum over $M$, proving the theorem.
\end{proof}

\begin{rk}
Define $\bar {\mathcal X}$ to be the space consisting of all finite nonnegative $f\in L^1[0,1]$, modulo the relation $f\sim cf$ for $c>0$, and equipped with the complete metric $$d_{\bar{\mathcal{ X}}} \big( [f],[g]\big):= \bigg\|\frac{f}{\|f\|_{L^1}}-\frac{g}{\|g\|_{L^1}}\bigg\|_{L^1}.$$ Note by \eqref{t2} that in \eqref{t1} we may replace $d_\mathcal X$ by $d_{\bar{\mathcal X}}$ and the statement is still true. Therefore the remainder of the proof still goes through with $\mathcal X$ replaced by $\bar {\mathcal X}$, which then gives that $$\Bbb E \bigg[ \sup_{[\mu],[\nu]\in\mathcal X} d_{\bar{\mathcal X}}\big([z_N^\mu(0,\cdot)],[z_N^\nu(0,\cdot)]\big)\bigg] \leq Ce^{-cN}.$$ This gives a one-force-one-solution principle with respect to a stronger topology. This was the choice of topology in \cite{GK20}. In fact one has the following corollary in the much stronger Hölder norm. 
\end{rk}

\begin{cor}[Synchronization in Hölder norm]\label{1f1sy}Fix a space-time white noise on $(-\infty,0]\times[0,1]$ and let $z_N^\mu,z_N^\nu: [-N,0]\times [0,1]\to \Bbb R$ denote the solutions to the stochastic heat equation \eqref{mshe} driven by $\xi$ started from any deterministic $\mu,\nu \in \mathcal X$ at time $t=-N$. Then 
one has that 
$$\sup_{[\mu],[\nu]\in \mathcal X}d_\mathcal Y\big([z_N^\mu(0,\cdot)],[z_N^\nu(0,\cdot)]\big) \to 0,\;\;\;\;\;almost \; surely.$$ 
Here $\mathcal Y$ is the space defined in \eqref{yspace}.
\end{cor}

We remark that this proves \cite[Conjecture 1.1]{KM22} in a much stronger form, namely with a supremum over all possible initial data.

\begin{proof}

Let $\mathcal M_1[0,1]$ denote the set of probability measures on $[0,1]$. Let $d_{\text{Proh}}$ denote the Prohorov metric on $\mathcal M_1[0,1]$. Let $z^\mu$ denote the solution of the stochastic heat equation \eqref{mshe} started from initial measure $\mu$, and let $z^\mu(1,\cdot)$ denote the spatial process $x\mapsto z^\mu(1,x)$.

For each realization of $\xi$ define a random mapping $\Phi_\xi$ from $\mathcal M_1[0,1] \to C^\kappa[0,1]$ by sending $\mu\mapsto \log z^\mu(1,\cdot).$ Notice that for $\kappa<1/2,$
\begin{align}\sup_{\mu\in \mathcal M_1[0,1]} |z^\mu(1,x)-z^\mu(1,y)| &\leq \sup_{\mu\in \mathcal M_1[0,1]} \int_{[0,1]} |z_1(a,x)-z_1(a,y)|\mu(da) \leq \|\mathbf z(1,\cdot,\cdot)\|_{C^\kappa([0,1]^2)} |x-y|^{\kappa},\label{fra}
\end{align}uniformly over all $x,y\in [0,1]$, where we are using the notation of Proposition \ref{nega}. This bound implies a.s. equicontinuity of the random family $(z^\mu(1,\cdot))_{\mu \in \mathcal M_1[0,1]}.$ If $\mu_n\to \mu$ weakly then it is clear that from continuity of $a\mapsto z_1(a,x)$ for all $x$ one has the almost sure convergence 
\begin{equation}\label{frb}z^{\mu_n}(1,x) = \int_{[0,1]} z_1(a,x)\mu_n(da)\to \int_{[0,1]}z_1(a,x)\mu(da) = z^\mu(1,x).
\end{equation}
Moreover, we also have that 
\begin{equation}\label{frc}\inf_{\mu\in \mathcal M_1[0,1]}\inf_{x\in[0,1]} z^\mu(1,x) = \inf_{\mu\in \mathcal M_1[0,1]}\inf_{x\in[0,1]} \int_{[0,1]} z_1(a,x)\mu(da) \geq \inf_{a,x\in[0,1]}z_1(a,x)>0.
\end{equation}
Combining \eqref{fra}, \eqref{frb}, \eqref{frc}, we can conclude that with probability 1, the map $\Phi_\xi$ is almost surely continuous from 
$$\big( \mathcal M_1[0,1], d_{\text{Proh}}\big) \to \big( C^\kappa[0,1],\|\cdot\|_{C^\kappa[0,1]}\big).$$ By compactness of $\big( \mathcal M_1[0,1], d_{\text{Proh}}\big)$, it is uniformly continuous, so we conclude that there exists a \textit{random} continuous increasing function $w_\xi : [0,\infty)\to [0,\infty)$ with $w_\xi(0)=0$ such that $$\|\log z^\mu(1,\cdot)-\log z^\nu(1,\cdot)\|_{C^\kappa[0,1]}\leq w_\xi\big(d_{\text{Proh}}(\mu,\nu)\big),$$ uniformly over all $\mu,\nu\in \mathcal M_1[0,1]$. In the notation of the corollary statement, this implies by the definition \eqref{xspace} of $d_\mathcal X$ that \begin{align}\notag\sup_{\mu,\nu\in\mathcal M_{\text{Borel}}[0,1]} \bigg\|\log\bigg(\frac{ z_N^\mu(0,\cdot)}{\int_{[0,1]}z_N^\mu(-1,\cdot) }\bigg)&-\log\bigg(\frac{ z_N^\nu(0,\cdot)}{\int_{[0,1]}z_N^\nu(-1,\cdot) }\bigg)\bigg\|_{C^\kappa[0,1]} \\&\leq\notag \sup_{\mu,\nu\in\mathcal M_{\text{Borel}}[0,1]} w_\xi \bigg( d_{\text{Proh}} \bigg( \frac{ z_N^\mu(-1,\cdot)}{\int_{[0,1]}z_N^\mu(-1,\cdot) },\frac{ z_N^\nu(-1,\cdot)}{\int_{[0,1]}z_N^\nu(-1,\cdot) }\bigg)\bigg)\\ &= \notag \sup_{[\mu],[\nu]\in\mathcal X} w_\xi\big(d_\mathcal X \big([z_N^\mu(-1,\cdot)],[z_N^\nu(-1,\cdot)]\big)\big)\\&=\label{gn} w_\xi \bigg( \sup_{[\mu],[\nu]\in\mathcal X} d_\mathcal X \big([z_N^\mu(-1,\cdot)],[z_N^\nu(-1,\cdot)]\big)\bigg).\end{align}Here $\mathcal M_{\text{Borel}}[0,1]$ is the collection of all nonnegative and nonzero Borel measures on $[0,1]$. In the last bound we used the fact that $w_\xi$ is increasing and continuous.

Now for each realization of $\xi$ and each $N\in\Bbb N$, if we let $\mathcal G_N$ denote the (random) collection of Borel measures $\mu$ on $[0,1]$ such that $\int_{[0,1]}z_N^\mu(-1,\cdot)=1$, then every equivalence class $[\mu]\in \mathcal X$ has a representative that lies in $\mathcal G_N$. By the definition \eqref{yspace} of $d_\mathcal Y$, we have that $d_\mathcal Y([f],[g]) \leq 2\|\log f - \log g\|_{C^\kappa[0,1]},$ therefore 
\begin{align*}\sup_{[\mu],[\nu]\in\mathcal X}d_\mathcal Y\big( [z^\mu_N(-1,\cdot)],[z^\nu_N(-1,\cdot)]\big) &\leq 2 \sup_{\mu,\nu\in\mathcal G_N} \big\|\log z_N^\mu(0,\cdot)-\log z_N^\nu(0,\cdot)\big\|_{C^\kappa[0,1]} \\&\leq 2w_\xi \bigg( \sup_{[\mu],[\nu]\in\mathcal X} d_\mathcal X \big([z_N^\mu(-1,\cdot)],[z_N^\nu(-1,\cdot)]\big)\bigg),
\end{align*}
where we use \eqref{gn} and the definition of $\mathcal G_N$ in the last bound. Since $w_\xi(0)=0$ and $w_\xi$ is continuous, it follows that the latter quantity tends to zero almost surely as soon as we can show that \begin{equation}\sup_{[\mu],[\nu]\in\mathcal X} d_\mathcal X \big([z_N^\mu(-1,\cdot)],[z_N^\nu(-1,\cdot)]\big)\stackrel{N\to\infty}{\longrightarrow} 0 \;\;\;\;\;almost \;surely.\label{ast}\end{equation}
To prove this, let $C,c$ be as in Theorem \ref{1f1s2}, and note by that theorem and Markov's inequality that 
\begin{align*}\Bbb P \bigg( \sup_{[\mu],[\nu]\in\mathcal X} d_\mathcal X \big([z_N^\mu(-1,\cdot)],[z_N^\nu(-1,\cdot)]\big)>e^{-cN/2}\bigg)& \leq e^{cN/2} \Bbb E\bigg[ \sup_{[\mu],[\nu]\in\mathcal X} d_\mathcal X \big([z_N^\mu(-1,\cdot)],[z_N^\nu(-1,\cdot)]\big)\bigg]\\&\leq Ce^{-cN/2}.
\end{align*}Therefore we immediately conclude \eqref{ast} by Borel-Cantelli.
\end{proof}

\begin{rk}\label{solflow}
    Note that the proof of Corollary \ref{1f1sy}  contains the proof of the following statement. For each realization of the driving noise $\xi$, if we define the \textbf{time-one solution flow} of the KPZ equation to be the random mapping $\Psi_\xi:\mathcal X \to \mathcal Y$ given by $[\mu]_\mathcal X \mapsto [\log z^\mu(1,\cdot)]_\mathcal Y$ then $\Psi_\xi$ is necessarily a continuous map for almost every $\xi$. This is an extremely strong property when combined with the compactness of $\mathcal X$. Indeed it implies that the image $\Psi_\xi(\mathcal X)$ is a compact random subset of $\mathcal Y$ which is not obvious at all at first glance. Roughly speaking it means that the solution flow for the KPZ equation has extremely strong contractive properties which are stronger even than the Neumann-boundary \textit{additive} stochastic heat equation on $[0,1]$, which is \eqref{kpz0} without the nonlinear term on the right side. The latter is a Gaussian object that is much easier to study, but it does \textbf{not} have a uniform spectral gap or a compactification of the state space as we see for open KPZ. In general these type of compactness statements can be seen as a form of ``coming down from infinity,'' see e.g. \cite{TW,MW} where related phenomena have been studied for the dynamical $\Phi^4_2$ and $\Phi^4_3$ equations respectively.
\end{rk}

Let us remark on some of the above results. Note that Theorem \ref{1f1s2} already implies the uniqueness of the invariant measure. Also note that both the proofs of Theorem \ref{1f1s2}, Corollary \ref{1f1sy}, and Remark \ref{solflow} all generalize to the case of a spatially colored noise, as well as periodic boundary. Neither proof required anything about the strong Feller property.

The next question to address is total variation convergence. So far in this section we have not used the strong Feller property or the fact that the noise is white in space. If we bring these properties into the discussion, then we can recover the geometric ergodicity results of the previous section using a different proof. Specifically we will now 
show the following discrete-time alternative to Theorem \ref{tv3} when one replaces the assumption of unique ergodicity with that of an asymptotic vanishing condition on differences. 

\begin{thm}\label{expo}
Let $\mathfrak P$ be a strong Feller Markov operator on some compact state space $\mathcal X$, and let $\mathfrak P_*$ denote the adjoint operator on measures. Assume that 
$\mathfrak P^N_*\delta_\phi - \mathfrak P^N_*\delta_\psi$ converges weakly to $0$ for all $\phi,\psi\in\mathcal X$, i.e., $$\lim_{N\to\infty}|\mathfrak P^NF(\phi) - \mathfrak P^NF(\psi)|=0$$ for all $\phi,\psi\in\mathcal X$ and all Lipschitz continuous $F:\mathcal X\to \Bbb R$. Then there exists $C,c>0$ such that $$\sup_{\phi,\psi\in\mathcal X} \|\mathfrak P^N_*\delta_\phi - \mathfrak P^N_*\delta_\psi\|_{TV} \leq Ce^{-cN}.$$
\end{thm}

In other words, compactness plus the the strong Feller property allows us to upgrade pointwise convergence of differences in the weak topology to \textit{uniform} convergence of differences in total variation, and moreover the uniform convergence is automatically exponentially fast. The continuous-time analogue of this result is clearly implied by Theorem \ref{tv3}. 
However 
we have formulated this result in discrete time, which therefore requires a separate proof. It is possible to give a spectral-theoretic proof of this statement just as in the case of Theorem \ref{tv3}, but instead we choose to go a different route via Arzela-Ascoli.

\begin{proof}
We break the proof into two steps.
\\
\\
\textit{Step 1. } First we establish the intermediate claim that 
\begin{equation}\label{inter}\lim_{N\to \infty}\sup_{\phi,\psi\in\mathcal X} \|\mathfrak P^N_*\delta_\phi - \mathfrak P^N_*\delta_\psi\|_{TV}=0.
\end{equation}Let $\mathcal M_1(\mathcal X)$ be the space of probability measures on $\mathcal X$, and equip $\mathcal M_1(\mathcal X)$ with the total variation norm. Fix some reference state $\phi_0\in\mathcal X$, and define $G_N:\mathcal X\to\mathcal M_1(\mathcal X)$ by $\phi\mapsto \mathfrak P_*^N (\delta_\phi-\delta_{\phi_0}).$ We know by assumption that the pointwise limit of $G_N$ equals the zero measure in the topology of weak convergence of measures. Now the goal is to upgrade this pointwise-weak convergence to total variation convergence that is uniform over all $\phi.$ To do this, we are going to apply Arzela-Ascoli to the maps $G_N$. 

First we claim that the collection $\{G_N\}_{N\ge 2}$ is equicontinuous from $\mathcal X\to\mathcal M_1(\mathcal X)$. To prove this, write
$$\big\| G_N(\phi) - G_N(\psi)\big\|_{TV} = \big\| \mathfrak P^N_*(\delta_\phi - \delta_\psi)\big\|_{TV} \leq \|\mathfrak P^2_*(\delta_\phi - \delta_\psi)\|_{TV},$$
where the second inequality follows from the fact that $N\ge 2$ and $\|\mathfrak P_*(\mu-\nu)\|_{TV}\leq \|\mu-\nu\|_{TV}$. By the ultra Feller property (see Definition \ref{ufq} and Proposition \ref{UFP}) we know $\phi \mapsto \mathfrak P_*^2 \delta_\phi$ is continuous from $\mathcal X\to \mathcal M_1(\mathcal X)$, and therefore uniformly continuous by compactness. Thus for all $\epsilon>0$ there exists some $\delta$ such that the right side is less than $\epsilon$ uniformly over all $\phi,\psi\in\mathcal X$ such that $d_\mathcal X(\phi,\psi)<\delta,$ which proves the equicontinuity.

Next, we need to prove pointwise compactness of $G_N$. For each $\phi$, we know by compactness of $\mathcal X$ that the collection $\{\mathfrak P^{N-2}_*\delta_\phi\}_{N\ge 2}$ is tight, and therefore precompact in the weak topology on $\mathcal M_1(\mathcal X)$. By the ultra Feller property, $\mathfrak P_*^2$ is continuous from the weak topology on $\mathcal M_1(\mathcal X)$ to the total variation topology on $\mathcal M_1(\mathcal X)$, therefore it follows that the collection $G_N(\phi)= \mathfrak P_*^2\big(\mathfrak P^{N-2}_*\delta_\phi\big)-\mathfrak P_*^2\big(\mathfrak P^{N-2}_*\delta_{\phi_0}\big)$ is pointwise relatively precompact in the TV topology for each $\phi$. 

Therefore by Arzela-Ascoli, any pointwise limit of $G_N$ in weak convergence is in fact a uniform limit in total variation distance, which establishes \eqref{inter} since $$\sup_{\phi,\psi\in\mathcal X} \|\mathfrak P^N_*\delta_\phi - \mathfrak P^N_*\delta_\psi\|_{TV} \leq \sup_{\phi\in\mathcal X} \|\mathfrak P^N_*\delta_\phi - \mathfrak P^N_*\delta_{\phi_0}\|_{TV}+\sup_{\psi\in\mathcal X} \|\mathfrak P^N_*\delta_{\phi_0} - \mathfrak P^N_*\delta_\psi\|_{TV}.$$
\textit{Step 2. } Next we need to show that \eqref{inter} necessarily happens exponentially fast, which will require a separate argument. Define $\mathcal M_0(\mathcal X)$ to be the Banach space consisting of all signed Borel measures $\rho$ on $\mathcal X$ with $\rho(\mathcal X)=0$, equipped with total variation norm. For a linear operator $T:\mathcal M_0(\mathcal X)\to \mathcal M_0(\mathcal X),$ define $\|T\|_{\text{op}}:=\sup_{\|\rho\|_{TV}=1} \|T(\rho)\|_{TV}$ to be the operator norm. Note that $\mathfrak P_*$ is a bounded linear operator from $\mathcal M_0(\mathcal X)\to \mathcal M_0(\mathcal X)$ with $\|\mathfrak P_*\|_{\text{op}} \leq 1.$

The Hahn-Jordan decomposition theorem says that any $\rho \in \mathcal M_0(\mathcal X)$ can be decomposed as $\rho_+-\rho_-$ where both $\rho_+,\rho_-$ are nonnegative measures and $\rho_+(\mathcal X) = \rho_-(\mathcal X).$ If $\|\rho\|_{TV}=2$ then $\rho_+$ and $\rho_-$ are probability measures and we can write 
\begin{align*}
    \rho &= \rho_+-\rho_- = \int_\mathcal X \delta_\phi \rho_+(d\phi) - \int_\mathcal X \delta_\psi \rho_-(d\psi) =\int_\mathcal X\int_\mathcal X \big( \delta_\phi-\delta_\psi\big) \rho_+(d\phi)\rho_-(d\psi), 
\end{align*}
so that 
$$\mathfrak P_*^N\rho = \int_\mathcal X\int_\mathcal X \mathfrak P_*^N\big( \delta_\phi-\delta_\psi\big) \rho_+(d\phi)\rho_-(d\psi),$$ where these formal expressions should be interpreted by pairing these measures against continuous functions on $\mathcal X$, then moving the pairings inside the integrals accordingly. Consequently 
\begin{align*}
    \sup_{\|\rho\|_{TV}=2} \|\mathfrak P_*^N\rho\|_{TV} &\leq \sup_{\|\rho\|_{TV}=2} \int_\mathcal X\int_\mathcal X \big\|\mathfrak P_*^N\big( \delta_\phi-\delta_\psi\big)\big\|_{TV} \rho_+(d\phi)\rho_-(d\psi)\\& \leq \sup_{\phi,\psi\in\mathcal X}\big\|\mathfrak P_*^N\big( \delta_\phi-\delta_\psi\big)\big\|_{TV} \stackrel{N\to\infty}{\longrightarrow} 0,
\end{align*}
where we used the fact that $\rho_+,\rho_-$ are probability measures in the second inequality, and then \eqref{inter} in the convergence statement.

We have proved that $\|\mathfrak P_*^N\|_{\text{op}}\to 0$ as $N\to\infty$. Choose some $N_0\in \Bbb N$ such that $\|\mathfrak P_*^{N_0}\|_{\text{op}} \leq 1/2.$ Then for all $k \in \Bbb N$ and all $r\in \{0,...,N_0-1\},$ we have $$\|\mathfrak P_*^{kN_0+r}\|_{\text{op}} \leq\|\mathfrak P_*^{kN_0}\|_{\text{op}}\|\mathfrak P_*^r\|_{\text{op}} \leq \|\mathfrak P_*^{N_0}\|^k_{\text{op}}\|\mathfrak P_*\|_{\text{op}}^r\leq 2^{-k}\cdot 1,$$ where we use the sub-multiplicativity of the operator norm several times. This exponential bound on the operator norm clearly implies the claim.
\end{proof}

We remark that neither compactness nor the strong Feller property were needed in the second step of the proof. In other words, if $\mathfrak P$ is \textit{any} Feller Markov operator on \textit{any} Polish space satisfying \eqref{inter}, then the convergence is automatically exponential. Compactness plus strong Feller just give a sufficient condition for \eqref{inter} to hold.

On the other hand, the first step of the argument also generalizes to show that if $\mathcal X$ is some Polish space (not necessarily compact) and if $K\subset \mathcal X$ is a compact subset, then for any strong Feller Markov operator $\mathfrak P$ on $\mathcal X$ such that $\mathfrak P_*^N\big( \delta_\phi-\delta_\psi\big)$ converges weakly to zero for all $\phi,\psi\in K$, and such that $\{\mathfrak P_*^N \delta_\phi\}_{N\ge 1}$ is tight for some $\phi\in K$, one actually has $$\lim_{N\to\infty}\sup_{\phi,\psi\in K} \big\|\mathfrak P_*^N\big( \delta_\phi-\delta_\psi\big)\big\|_{TV}=0,$$ though not necessarily exponentially fast.

\begin{cor}[Another proof of geometric ergodicity]\label{bob}
Let $\mathfrak P$ be a strong Feller Markov operator on some compact state space $\mathcal X$, and let $\mathfrak P_*$ denote the adjoint operator on measures. Assume that 
$\mathfrak P^N_*\delta_\phi - \mathfrak P^N_*\delta_\psi$ converges weakly to $0$ for all $\phi,\psi\in\mathcal X$. Then $\mathfrak P$ admits a unique invariant measure $\rho$ and one has that $$\sup_{\phi\in\mathcal X} \|\mathfrak P^N_*\delta_\phi - \rho\|_{TV} \leq Ce^{-cN},$$
In particular this is true for the Markov operator of open KPZ on 
the space $\mathcal X$ given in \eqref{xspace}.
\end{cor}

\begin{proof}
Compactness of $\mathcal X$ guarantees the existence of at least one invariant probability measure $\rho$ via Krylov-Bogoliulov. For all $N\in\Bbb N$ the invariant measure satisfies $$\rho = \int_\mathcal X \mathfrak P_*^N \delta_\psi\; \rho(d\psi),$$ therefore we see that 
\begin{align*}
\sup_{\phi\in\mathcal X} \|\mathfrak P^N_*\delta_\phi - \rho\|_{TV} &= \sup_{\phi\in\mathcal X} \bigg\|\mathfrak P^N_*\delta_\phi - \int_\mathcal X  \mathfrak P_*^N \delta_\psi\;\rho(d\psi)\bigg\|_{TV} \\ &= \sup_{\phi\in\mathcal X} \bigg\|\int_\mathcal X \big(\mathfrak P^N_*\delta_\phi -  \mathfrak P_*^N \delta_\psi\big)\;\rho(d\psi)\bigg\|_{TV} \\&\leq \sup_{\phi\in\mathcal X} \int_\mathcal X \big\| \mathfrak P^N_*\big(\delta_\phi -  \delta_\psi\big)\big\|_{TV}\rho(d\psi) \\&\leq \sup_{\phi,\psi\in\mathcal X} \|\mathfrak P^N_*\big(\delta_\phi - \delta_\psi\big)\|_{TV} \\ &\leq Ce^{-cN},
\end{align*}
where we use Theorem \ref{expo} in the last bound.

For the special case of KPZ, Corollary \ref{SFP} shows that $\mathfrak P$ is strong Feller on $\mathcal X$. Moreover, Theorem \ref{1f1s2} clearly guarantees $\mathfrak P^N_*\delta_\phi - \mathfrak P^N_*\delta_\psi$ converges weakly to $0$ for all $\phi,\psi\in\mathcal X.$
\end{proof}

\end{document}